\documentclass[a4paper,12pt]{amsart}
\usepackage{amssymb, amsmath,amscd,amsthm}
\usepackage{textcomp}
\usepackage{pifont}
\usepackage{amsfonts}
\usepackage{latexsym}
\usepackage{bm}
\usepackage{enumitem}
\usepackage[utf8x]{inputenc}
\usepackage{verbatim}
\usepackage{upgreek}
\usepackage{mathtools}
\usepackage[all]{xy}
\usepackage{xcolor}
\usepackage{multicol}
\usepackage{blindtext}
\usepackage{graphicx}
\usepackage{stmaryrd}
\usepackage{comment}

\newcommand{\R}{\mathbb{R}}
\newcommand{\Z}{\mathbb{Z}}
\newcommand{\C}{\mathbb{C}}
\newcommand{\Q}{\mathbb{H}}
\newcommand{\K}{\mathbb{K}}
\DeclareMathOperator{\A}{\mathrm{Aff}}
\DeclareMathOperator{\I}{\mathrm{Iso}}
\DeclareMathOperator{\G}{\mathrm{GL}}
\DeclareMathOperator{\ort}{\mathrm{O}}
\newcommand{\Id}{\mathrm{Id}}

\theoremstyle{plain}
\newtheorem{thm}{Theorem}[section]
\newtheorem{lemma}[thm]{Lemma}
\newtheorem{proposition}[thm]{Proposition}
\newtheorem{corollary}[thm]{Corollary}

\theoremstyle{definition}
\newtheorem{definition}[thm]{Definition}
\newtheorem{remark}[thm]{Remark}

\newtheorem*{ack}{Acknowledgements}

\author[Karla Garc\'{\i}a]{Karla Garc\'{\i}a}
\address{Departamento de Matem\'aticas, Facultad de Ciencias, Universidad Nacional Aut\'onoma de M\'exico (UNAM), Ciudad Universitaria, CP 04510 M\'exico. }
\email{ohmu@ciencias.unam.mx}
\thanks{K. Garc\'{\i}a was supported by a postdoctoral fellowship from UNAM-DGAPA}

\author[Oscar Palmas]{Oscar Palmas}
\address{Departamento de Matem\'aticas, Facultad de Ciencias, Universidad Nacional Aut\'onoma de M\'exico (UNAM), Ciudad Universitaria, CP 04510 M\'exico. }
\email{oscar.palmas@ciencias.unam.mx}
\thanks{O. Palmas acknowledges the support of Conacyt-Mexico SNII 16167 and UNAM-Mexico, Project DGAPA-PAPIIT IN101322.}

\title{Moduli spaces of flat Riemannian metrics on $4$-dimensional closed manifolds}

\begin{document}

\maketitle

\markright{ \uppercase{Moduli spaces of flat Riemannian metrics on closed manifolds}}

\begin{abstract}
We give the algebraic and topological description of the moduli spaces of flat metrics for the $4$-dimensional closed flat manifolds with two or three generators in their holonomy, which completes the analysis made in \cite{karla}.    
\end{abstract}

\section{Introduction}\label{sec:intro}

A fundamental question in Riemannian geometry is whether there exists a deformation of a metric on a given manifold $M$ preserving certain curvature conditions. For this, is important to study the space of all such metrics on $M$ and also its moduli space, i.e., its quotient by the full diffeomorphism group of $M$, acting by pulling back metrics. 

Tuschmann and Wiemeler developed a convenient method to study moduli spaces of Riemannian metrics with non-negative Ricci curvature on manifolds diffeomorphic to the product of a closed simply connected manifold with a torus \cite{tuschwiem}. They showed that the problem of understanding the topology of these moduli spaces may be reduced to the study of the moduli space of flat Riemannian metrics on the torus, that is, metrics such that the corresponding sectional curvature vanishes identically.

The moduli space of flat metrics on a fixed closed manifold $M$, denoted here by $\mathcal{M}_{flat}(M)$, is the quotient of the Teichm\"uller space of flat metrics $\mathcal{T}_{flat}(M)$, diffeomorphic to an Euclidean space, under the action of a discrete mapping class group, denoted by $\mathcal{N}_{\pi}$. For details, see \cite{bettiol}. 
 
A study of $\mathcal{M}_{flat}(M)$ for $3$-dimensional closed manifolds was developed by Kang in \cite{kang}. The first named author studied in \cite{karla} the topology of $\mathcal{M}_{flat}(M)$ for $3$-dimensional closed manifolds, giving also an algebraic description of the moduli spaces for the $4$-dimensional closed flat manifolds with a single generator in their holonomy. In the present paper we complete this analysis, giving
the algebraic description of the moduli spaces of flat metrics for the $4$-dimensional closed manifolds with two or three generators in their holonomy, classified according to the types of their isotypic components. The isotypic components are the direct sum of isomorphic irreducible subspaces of the orthogonal representation of the holonomy of the manifolds on $\R^n$. Details will be given in section \ref{sec:teichmuller}.

Let us introduce some notation. Remember that $\G(n, \Z)$ is the Lie group of all integer matrices with determinant $\pm 1$. Consider a matrix of the form 
$ X= \begin{psmallmatrix}
a & b \\ c & d \end{psmallmatrix} $ and denote
\begin{align*}
\Gamma_0 (2)^t & = \{ X  \in \G(2, \Z) \mid b \equiv 0 \mod 2\},\\  
\Gamma (2) & = \{ X \in \G(2, \Z) \mid c,b \equiv 0 \; \text{and} \; a,d \equiv 1 \mod 2\}.
\end{align*}
We use the notation for flat manifolds given in \cite{lambert}. We can state our main result as follows.


\begin{thm} \label{modulimoregen}
The moduli spaces of flat metrics on $4$-dimensional closed manifolds with two or more generators in their holonomy are given as follows:
\begin{enumerate}
\item For the case of four different $1$-dimensional $\R$ type isotypic components,
\begin{enumerate}
\item When $M$ is $O^4_{i}$, for $i \in \{ 9,11,12,15,16,17\} $, or $N^4_{45}$, we have \\
$ \mathcal{M}_{flat}(M) \cong (\R^{+})^3 \times [0, \infty  ) $.
\item When $M=O^4_{14}$, we have $ \mathcal{M}_{flat}(M) \cong (\R^{+})^2 \times [0, \infty  )^2 $.
\item When $M$ is $O^4_{10}$, $O^4_{13}$, or 
 $N^4_{i}$, for $i\in \{ 30,\dots,39\}\cup\{46\}$, we have \\
 $ \mathcal{M}_{flat}(M) \cong (\R^{+})^4$.  
\end{enumerate}
\item For the case of two different $1$-dimensional $\R$ type isotypic components, i.e., when $M$ is $O^4_{26}$, $O^4_{27}$ or $N^4_{43}$, we have $ \mathcal{M}_{flat}(M) \cong (\R^{+})^2$.
\item For the case of two different $1$-dimensional $\R$ type isotypic components and one $2$-dimensional $\R$ type isotypic component,
\begin{enumerate}
\item When $M=N^4_{22}$, we have \\ $ \mathcal{M}_{flat}(M) \cong (\text{O}(2)\backslash \text{GL}(2,\mathbb{R}) /\text{GL}(2,\mathbb{Z})) \times (\R^{+})^2 $.
\item When $M$ is $N^4_{i}$, for $i\in \{ 8,9\}\cup\{23,\dots,26 \}$, we have \\ $ \mathcal{M}_{flat}(M) \cong \left( \ort(2) \backslash \G(2, \R) / \Gamma_0(2)^t \right) \times (\R^{+})^2$. 
\item When $M$ is $N^4_{i}$, for $i\in \{ 4\}\cup\{10,\dots,13 \}$, we have \\ $ \mathcal{M}_{flat}(M) \cong \left( \ort(2) \backslash \G(2, \R) / \Gamma (2) \right) \times (\R^{+})^2$.
\item When $M$ is $N^4_{3}$, $N^4_{5}$ or $N^4_{6}$, let $X_1= X_2= \mathbb{H}^2 / \Gamma(2)\cdot \{ \left\langle \begin{psmallmatrix}
0 & -1 \\ 
1  & 0  \end{psmallmatrix}  \right\rangle \}^+$, then \\ 
$ \mathcal{M}_{flat}(M) \cong \left( X_1 \cup_{f_t} X_2 \right) \times (\R^{+})^2 \times [0, \infty ) $, with $t\in [0, \infty)$. When $t\neq 0$, the attaching map $f_t$ is the identity restricted to a path inside $X_1$. When $t=0$, the attaching map is the identity in the whole $X_1$. 
\item When $M=N^4_{7}$, let $X_1= X_2= \mathbb{H}^2 / \Gamma(2)^+$, then \\
$ \mathcal{M}_{flat}(M) \cong \left( X_1 \cup_{f_t} X_2 \right) \times (\R^{+})^2 \times [0, \infty ) $, with $t\in [0, \infty)$. When $t\neq 0$, the attaching map $f_t$ is the identity restricted to a path inside $X_1$. When $t=0$, the attaching map is the identity in the whole $X_1$.  
\item When $M=N^4_{44}$, we have \\ $ \mathcal{M}_{flat}(M) \cong \left( \ort(2) \backslash \G(2, \R) / \Gamma_0(2)^t \right) \times \R^{+} \times [0, \infty )$.
\end{enumerate}
\item For the case of two different $1$-dimensional $\R$ type isotypic components and one $1$-dimensional $\C$ type isotypic component, i.e., when $M$ is $O^4_{i}$, for $i \in \{ 18,\dots,25\} $, or $N^4_{i}$, for $i\in \{ 27,28,29,40,41,42,47 \}$, we have \\ $ \mathcal{M}_{flat}(M) \cong (\R^{+})^3$.  
\end{enumerate}

\end{thm}


Cases 3(d) and 3(e) are examples of moduli spaces of flat metrics where the action of $\mathcal{N}_{\pi}$ on $\mathcal{T}_{flat}(M)$ does not split as a product. Here the moduli space is homeomorphic to a product where one of the factors is a space attached to itself by  a curve.

For the following corollary we use the techniques given in \cite{karla} in order to study the topology of the moduli spaces of flat metrics.

\begin{corollary} \label{topmoregen}
The non-contractible moduli spaces of flat metrics on $4$-dimensional closed manifolds with two or more generators in their holonomy are given in items (3)(b)-(f) of Theorem \ref{modulimoregen}. More precisely:

\renewcommand{\labelenumi}{(\alph{enumi})}
\begin{enumerate}
\setcounter{enumi}{1}
\item When $M$ is $N^4_{i}$, for $i\in \{ 8,9\}\cup\{23,\dots,26 \}$, we have $ \mathcal{M}_{flat}(M) \cong \mathbb{S}^1 \times (\R^{+})^4$. 
\item When $M$ is $N^4_{i}$, for $i\in \{ 4\}\cup\{10,\dots,13 \}$, we have \\ $ \mathcal{M}_{flat}(M) \cong 3\text{-punctured sphere} \times (\R^{+})^3$.
\item When $M$ is $N^4_{3}$, $N^4_{5}$ or $N^4_{6}$, we have $ \mathcal{M}_{flat}(M) \cong \left( X_1 \cup_{f_t} X_2 \right) \times (\R^{+})^2 \times [0, \infty ) $, with $X_1=X_2 \cong \mathbb{S}^1 \times \R $. 
\item When $M=N^4_{7}$, we have $ \mathcal{M}_{flat}(M) \cong \left( X_1 \cup_{f_t} X_2 \right) \times (\R^{+})^2 \times [0, \infty ) $ with \\ 
$X_1= X_2\cong 3$-punctured sphere.   
\item When $M=N^4_{44}$, we have $ \mathcal{M}_{flat}(M) \cong \mathbb{S}^1 \times (\R^{+})^3 \times [0, \infty )$.
\end{enumerate}

All other moduli spaces from the list in Theorem  \ref{modulimoregen} are contractible.  
\end{corollary}

This paper is organized as follows. In section \ref{sec:prelim} we give some preliminary information about closed flat manifolds and Bieberbach groups. In section \ref{sec:teichmuller} we describe the Teichm\"uller space of flat metrics for the $4$-dimensional closed manifolds with two or more generators in their holonomy. In sections \ref{sec:isotypicreal} and \ref{sec:isotypic1complex} we study the matrix part of the normalizer of the respective Bieberbach group of each $4$-dimensional closed manifold. In section \ref{sec:subgroups} we analyze certain types of double quotients, which are used to study the topology of the moduli spaces of flat metrics. Finally, in section \ref{sec:moduli} we prove our main  results.

\begin{ack}
We thank Motiejus Valiunas for pointing out a mistake, which is explained in the remark of section \ref{sec:teichmuller}, and for useful conversations.
\end{ack}

\section{Preliminaries}\label{sec:prelim}

A {\bf flat manifold} is a Riemannian manifold which admits a metric of zero sectional curvature, called a {\bf flat metric}. These manifolds are basically described by their fundamental group, which turns out to be a Bieberbach group.

\begin{definition}
A {\bf Bieberbach group $\mathbf{\pi}$} is a discrete subgroup of the isometries of $\R^n$, that is torsion-free and such that $\R^n/\pi$ is compact. 
\end{definition}

We have that $(\R^n/\pi, \sigma)$, with $\pi$ a Bieberbach group and $\sigma$ the metric induced from the usual metric of $\mathbb{R}^n$, is a closed flat manifold. Conversely, let $M$ be a closed manifold with a flat metric $g$, then its universal cover with the metric induced from $g$ is isometric to $\mathbb{R}^n$ with the usual metric. In other words, $\R^n$ with the usual metric is a Riemannian covering of $(M,g)$ and we consider its group of deck transformations, denoted by $\pi$. Then $(M,g)$ is isometric to $(\R^n/\pi, \sigma)$, where $\pi$ is a Bieberbach group. Therefore a closed flat manifold is represented by its Bieberbach group and the Bieberbach theorems describe important properties about them. One of these properties is that two closed flat manifolds with isomorphic fundamental groups are affinely equivalent. For more details, see \cite{charlap} or \cite{wolf}. 

The group of affine transformations of $\R^n$, denoted by $\A(n)$, has the structure of a semidirect product $\A(n) = \G(n, \R) \ltimes \R^n $. The group of isometries of $\R^n$ denoted by $\I(n)$, also has the structure of a semidirect product, $\I(n)= \ort(n)\ltimes \mathbb{R}^n $. 

Consider the projection homomorphism
$$ 
\begin{array}{cccc}
\tau\colon & \A(n) & \rightarrow & \G(n, \R), \\ 
      &   (A,v)       & \mapsto     &    A .  
\end{array}      $$

\begin{definition}
Let $\pi$ be a Bieberbach group. The {\bf holonomy} of $\pi$ is the subgroup of $\G(n, \R)$ given by $H_{\pi}\coloneqq  \tau (\pi)$. 
\end{definition}

Thanks to the Bieberbach theorems, the space of flat metrics on a fixed closed manifold $M$ can be described in terms of their Bieberbach groups as follows: For a flat metric $g$ on $M$ we get a Bieberbach group $\pi$ such that $M=\R^n/\pi$, and if we consider another flat metric $g'$ on $M$ we get another Bieberbach group $\pi'$, where $\pi' = \gamma \pi \gamma^{-1} $ for some $\gamma \in \A(n)$. Then the space of flat metrics is $\{ \gamma \in \A(n) \mid  \gamma \pi \gamma^{-1} \subset \I(n) \}$. To define the Teichm\"uller space of flat metrics we just need to consider the matrix part of the previous space.

\begin{definition}
The {\bf Teichm\"uller space of flat metrics} of the flat manifold $M=\R^n/\pi$ is the orbit space $\ort(n) \backslash \{ A \in \G(n, \R) \mid A H_{\pi} A^{-1} \subset \ort(n) \}$ of the left $\ort(n)$-action by matrix multiplication.
\end{definition}

Now, we want to distinguish isometry classes of flat metrics. If we have two metrics $g$ and $g'$ on $M$, where $\pi '= (A,v) \pi (A,v)^{-1}$, we know that they are isometric if and only if $A=BC$, where $B\in \ort(n)$ and $C\in  \tau ( \mathrm{N}_{\A(n)} (\pi))$, $\mathrm{N}_{\A(n)} (\pi)$ being the normalizer of $\pi$ in $\A(n)$, see \cite{wolf2}. Since the moduli space $\mathcal{M}_{flat}(M)$ of a closed flat manifold $M= \R^n/{\pi}$ is defined as the set of isometry classes of flat metrics on $M$, then we obtain the following description.    

\begin{proposition} [\protect{\cite{bettiol}}, {Proposition 4.3}] \label{moduli}
\begin{equation} \label{modulieq}
\mathcal{M}_{flat}(M) = \mathcal{T}_{flat}(M)/ \mathcal{N}_{\pi},  \quad \text{where} \quad \mathcal{N}_{\pi} \coloneqq \tau ( \mathrm{N}_{\A(n)} (\pi)) . 
\end{equation}
\end{proposition}

$\mathcal{N}_{\pi}$ is discrete and the action is not necessarily free, so  $\mathcal{M}_{flat}(M)$ may have singularities. For more information about $\mathcal{M}_{flat}(M) $ see \cite{karla}. 

In order to describe in a more precise way $\mathcal{M}_{flat}(M)$ we use the list of the affine equivalence classes of closed flat manifolds, which is in bijective correspondence with the list of the affine conjugacy classes of Bieberbach groups in $\I(n)$. This is known for low values of the dimension of the manifolds. The $4$-dimensional case was studied by several authors such as Calabi \cite{calabi}, Levine \cite{levine}, Brown \cite{brown}, and Hillman \cite{hillman}, but we found the classification given by Lambert in  \cite{lambert} more suitable for geometers to use. The notation and representation we use from Lambert's list will be given in sections \ref{sec:isotypicreal} and \ref{sec:isotypic1complex}.    


\section{The Teichm\"uller space of flat metrics}\label{sec:teichmuller}

Since we are using the description of the moduli space of flat metrics given in \eqref{modulieq}, we start by analyzing the Teichm\"uller space of flat metrics. We apply Theorem B of  \cite{bettiol}, where it is used the decomposition of $\R^n$ into $H_{\pi}$-isotypic components. 

Recall that a nonzero invariant subspace $V$ of $\R^n$ is {\bf irreducible} if it does not contain any proper invariant subspace, or, equivalently, if every nonzero element of the vector space End$_{H_{\pi}}(V)$ of linear equivariant endomorphisms of $V$ is an isomorphism. In this situation, End$_{H_{\pi}}(V)$ is an associative real division algebra and hence isomorphic to one of $\R$, $\C$, or  $\Q$. The irreducible subspace $V$ is called of {\bf real}, {\bf complex}, or {\bf quaternionic} type, according to the isomorphism type of End$_{H_{\pi}}(V)$.

Consider the decomposition of the orthogonal $H_{\pi}$-representation into irreducible subspaces, $\R^n = V_{1,1} \oplus \cdots \oplus V_{1,m_1} \oplus \cdots \oplus V_{l,1} \oplus \cdots \oplus V_{l,m_l}$, where each $V_{i,j}$ is irreducible and $V_{i,j}$ is isomorphic to $V_{i',j'}$ if and only if $i=i'$, so that $W_i= \oplus_{j=1}^{m_i} V_{i,j}$, $i=1, \dots,l$ are the so-called {\bf isotypic components}. 

\begin{proposition} \label{teich}
The Teichm\"uller spaces $\mathcal{T}_{flat}(M)$ for the $4$-dimensional closed flat manifolds with two or more generators in their holonomy are the following:
\begin{enumerate}
\item When $M$ has four different $1$-dimensional $\R$ type isotypic components. This is the case of $O^4_{i}$, with $i \in \{ 9,\dots,17\}$, and $N^4_{i}$, with $i \in \{ 30,\dots,39\}\cup\{45, 46\}$. Then
$$ \mathcal{T}_{flat}(M) \cong \prod_{i=1}^4 \frac{\G(1, \R)}{\ort(1)} \cong (\R^+)^4.$$
\item When $M$ has two different $1$-dimensional $\R$ type isotypic components. This is the case of $O^4_{26}$, $O^4_{27}$ and $N^4_{43}$. Then
$$ \mathcal{T}_{flat}(M) \cong \prod_{i=1}^2 \frac{\G(1, \R)}{\ort(1)} \cong (\R^+)^2.$$
\item When $M$ has two different $1$-dimensional $\R$ type isotypic components and one $2$-dimensional $\R$ type isotypic component. This is the case of $N^4_{i}$, with $i \in \{ 3,\dots,13\}\cup\{22,\dots,26\}\cup\{ 44\}$. Then
$$ \mathcal{T}_{flat}(M) \cong \frac{\G(2, \R)}{\ort(2)} \times \prod_{i=1}^2 \frac{\G(1, \R)}{\ort(1)} \cong \frac{\G(2, \R)}{\ort(2)} \times (\R^+)^2.$$
\item When $M$ has two different $1$-dimensional $\R$ type isotypic components and one $1$-dimensional $\C$ type isotypic component. This is the case of $O^4_{i}$, with $i \in \{ 18,\dots,25\}$, and $N^4_{i}$, with $i \in \{ 27,28,29,40,41,42,47\}$. Then
$$ \mathcal{T}_{flat}(M) \cong  \prod_{i=1}^2 \frac{\G(1, \R)}{\ort(1)} \times \frac{\G(1, \C)}{\mathrm{U}(1)}  \cong  (\R^+)^2 \times \frac{\G(1, \C)}{\mathrm{U}(1)} .$$
  
\end{enumerate}
\end{proposition}

\begin{proof}
We identify the isotypic components of the representations of each Bieberbach group $\pi$ from the list in \cite{lambert}, which are given in several tables later. Then, use Theorem B of  \cite{bettiol} to describe the Teichm\"uller space.
\end{proof}

\begin{remark}
Theorem B of  \cite{bettiol} splits the Teichm\"uller space nicely into products. 
One may try to split the moduli space of flat metrics into products
\[
\prod_{i=1}^l \ort(m_i, \K_i) \backslash \G(m_i, \K_i) / \Gamma_i
\]
for each different isotypic component, but this is not possible. For example, take the $3$-dimensional closed flat manifold $G_6$, whose Teichm\"uller space is $\mathcal{T}_{flat}(G_6)= (\ort(1) \backslash \G(1, \R) )^3 \approx (\R^+)^3$. The action of $\mathcal{N}_{\pi}$ does not split into each factor; in fact, it is permutation of the coordinates in $(\R^+)^3$.
\end{remark}

\section{Manifolds whose isotypic components are all of real type}\label{sec:isotypicreal}

Here we describe the matrix part of the normalizer for the $4$-dimensional closed flat manifolds whose isotypic components are all of real type. They are the manifolds with holonomy $\Z_2^k$, two manifolds with holonomy $\mathrm{A}_4$ and one manifold with holonomy $\mathrm{A}_4 \times \Z_2$, where $\mathrm{A}_4$ is the alternating group on four letters.

The following lemmas are useful to analyze the matrix part of the normalizer.

\begin{lemma} \label{diag}
Let $\pi$ a Bieberbach group with holonomy $\Z_2^k$ for some $ k \in \mathbb{N}$. Then the matrices generating the holonomy may 
be represented by diagonal matrices with entries $\pm 1$. 
\end{lemma}

\begin{proof} 
Since the holonomy is the product of $k$ copies of $\Z_2$, the matrices $A \in H_{\pi}$ are such that $A^2 = \Id$ and we can diagonalize $A$ with eigenvalues $\pm 1$. Since for any other $B \in H_{\pi}$, $AB$ is again a reflection, then $AB=BA$ and $B$ will be a diagonal matrix for the component that was diagonalized for $A$. If we still have some $C \in H_{\pi}$ which is not diagonal, we repeat the same process.         
\end{proof}

\begin{lemma} \label{notinv}
Let $A$ and $B$ diagonal matrices with entries $\pm 1$. If $A$ has $k$ entries $-1$ and $B$ has $k+r$ entries $-1$, with $r \geq 1$, then there is no matrix in $X \in \G( n , \R)$ such that $XA=BX$.  
\end{lemma}

To prove the lemma we can use the property that similar matrices have the same characteristic polynomial.  

We denote a diagonal matrix with their corresponding entries in the diagonal as $\mathrm{diag}( a,b,c,d)$. We shall denote by $\{e_i\}$ the vectors of the standard basis of $\mathbb{R}^4$ and the basic translations of $\mathbb{R}^4$ by $t_i= (\Id, e_i) $. 

We have two different cases: when all the different isotypic components of real type are $1$-dimensional, and when we have one isotypic component of real type of dimension $2$ and the other two isotypic components are $1$-dimensional.

\subsection{All different $1$-dimensional isotypic components of real type.}

We will start with the manifolds with holonomy $H_{\pi}= \Z_2^2$. The Bieberbach groups of these closed flat manifolds are in table \ref{table:1}. 


\renewcommand{\arraystretch}{1.2}
\begin{table} 
\begin{tabular}{|l|c|c|} 
\hline
$\mathbf{O^4_9}$ &    $\pi = \langle t_1, t_2, t_3, t_4, \alpha, \beta \rangle $       &  $\alpha  = \left( \mathrm{diag}(-1,1,-1,1) , \frac{1}{2}e_2 \right) $\\ 
      &        &  $\beta= \left( \mathrm{diag}(-1,-1,1,1) , \frac{1}{2}e_4 \right)$  \\
 \hline
$\mathbf{O^4_{10}}$ &  $\pi = \langle t_1, t_2, t_3, t_4, \alpha, \beta \rangle $       & $\alpha= \left( \mathrm{diag}(1,-1,-1,1) , \frac{1}{2}e_1 + \frac{1}{2}e_2 \right)$  \\
      &        &  $\beta= \left( \mathrm{diag}(-1,-1,1,1) , \frac{1}{2}e_4 \right)$  \\
 \hline
$\mathbf{O^4_{11}}$ &  $\pi = \langle t_1, t_2, t_3, t_4, \alpha, \beta \rangle $       & $\alpha= \left( \mathrm{diag}(-1,-1,1,1) , \frac{1}{2}e_3 \right)$  \\
      &        &  $\beta= \left( \mathrm{diag}(1,-1,-1,1) ,\frac{1}{2}e_1 + \frac{1}{2}e_2 + \frac{1}{2}e_4 \right)$  \\
\hline      
$\mathbf{O^4_{12}}$ &  $\pi = \langle s_1, s_2, t_3, t_4, \alpha, \beta \rangle $       & $\alpha= \left( \mathrm{diag}(-1,-1,1,1) , \frac{1}{2}e_3 \right)$  \\
      & $s_1= (\Id , e_1 -e_2 )$     &  $\beta= \left(  \mathrm{diag}(1,-1,-1,1) , \frac{1}{2}e_4 \right)$  \\
    &  $s_2= (\Id , e_1 + e_2 )$  &       \\
\hline      
$\mathbf{O^4_{13}}$ &  $\pi = \langle s_1, s_2, s_3, t_4, \alpha, \beta \rangle $       & $\alpha= \left(  \mathrm{diag}(-1,1,-1,1) , \frac{1}{4}e_1 + \frac{1}{4}e_3 + \frac{1}{2}e_4 \right)$  \\
      &  $s_1=( \Id, \frac{1}{2}(e_2 - e_3))$  &  $\beta= \left(  \mathrm{diag}(1,1,-1,-1) ,\frac{1}{4}e_1 - \frac{1}{4}e_2 + \frac{1}{2}e_4 \right)$  \\
    & $s_2=( \Id, \frac{1}{2}(e_2 + e_3))$ &     \\
    & $s_3=( \Id, \frac{1}{2}(e_1 - e_2))$ &     \\   
\hline      
$\mathbf{O^4_{14}}$ &  $\pi = \langle t_1, t_2, t_3, t_4, \alpha, \beta \rangle $       & $\alpha= \left( \mathrm{diag}(1,-1,-1,1) , \frac{1}{2}e_1 + \frac{1}{2}e_2 \right)$  \\
      &        &  $\beta= \left(  \mathrm{diag}(-1,1,-1,1) ,-\frac{1}{2}e_1 + \frac{1}{2}e_2 - \frac{1}{2}e_3 \right)$  \\
\hline      
$\mathbf{O^4_{15}}$ &  $\pi = \langle t_1, t_2, t_3, s_1, \alpha, \beta \rangle $       & $\alpha= \left(  \mathrm{diag}(1,-1,1,-1) , \frac{1}{2}e_3 \right)$  \\
      &  $s_1=( \Id, \frac{1}{2}(e_1 + e_2 + e_3 + e_4))$  &  $\beta= \left( \mathrm{diag}(1,-1,-1,1) , \frac{1}{2}e_1  \right)$  \\
\hline      
$\mathbf{O^4_{16}}$ &  $\pi = \langle s_1, s_2, s_3, t_4, \alpha, \beta \rangle $       & $\alpha= \left(  \mathrm{diag}(-1,-1,1,1) , -e_1 - e_2 + \frac{1}{2}e_3 + \frac{1}{2}e_4  \right)$  \\
      &  $s_1=( \Id, \frac{1}{2}(e_1 + e_2 - e_3))$  &  $\beta= \left(  \mathrm{diag}(-1,1,1,-1) , \frac{1}{2}e_2  \right)$  \\
    & $s_2=( \Id, \frac{1}{2}(-e_1 + e_2 + e_3))$ &    \\
    & $s_3=( \Id, \frac{1}{2}(e_1 - e_2 + e_3))$ &     \\   
\hline      
$\mathbf{O^4_{17}}$ &  $\pi = \langle s_1, s_2, t_3, t_4, \alpha, \beta \rangle $       & $\alpha= \left(  \mathrm{diag}(-1,1,1,-1) , e_1 + \frac{1}{2}e_3 \right)$  \\
      & $s_1= (\Id , e_1 -e_2 )$     &  $\beta= \left(  \mathrm{diag}(1,1,-1,-1) , -e_1- \frac{1}{2}e_3 - \frac{1}{2}e_4 \right)$  \\
    &  $s_2= (\Id , e_1 + e_2 )$  &      \\
\hline      
\end{tabular}
\caption{Bieberbach groups with holonomy $\Z_2^2$ and different $1$-dimensional isotypic components of real type. See pages 42-44 in \cite{lambert}.}
\label{table:1}
\end{table}


Consider $\mathcal{S}_3$ the permutation group of 3 letters. 

\begin{proposition} \label{case 1}
The matrix part of the normalizer of $\pi$ in $\A(4)$ for the $4$-dimen\-sional closed flat manifolds with holonomy $\mathbb{Z}_2^2$ and all different $1$-dimensional isotypic components of real type is as follows:

\renewcommand{\arraystretch}{2}
\begin{tabular}{cl}
$\mathbf{O^4_{9}}$ & $\mathcal{N}_{\pi}= \left\{ \mathrm{diag}(\pm 1, \pm 1, \pm 1, \pm 1) , \begin{psmallmatrix}
0 & 0 & \pm 1 & 0  \\
0 & \pm 1 & 0 & 0 \\
\pm 1 & 0 & 0 & 0 \\ 
0 & 0 & 0     &  \pm 1     \end{psmallmatrix}  \right\}$, \\
$\mathbf{O^4_{10}}$, $\mathbf{O^4_{13}}$,    & $\mathcal{N}_{\pi}= \left\{ \mathrm{diag}(\pm 1, \pm 1, \pm 1, \pm 1)   \right\}$, \\
$\mathbf{O^4_{11}}$, $\mathbf{O^4_{12}}$, $\mathbf{O^4_{16}}$ & $\mathcal{N}_{\pi}= \left\{ \mathrm{diag}(\pm 1, \pm 1, \pm 1, \pm 1) , \begin{psmallmatrix}
0 & \pm 1 & 0 & 0  \\
\pm 1 & 0 & 0 & 0 \\
0 & 0 & \pm 1 & 0 \\ 
0 & 0 & 0     &  \pm 1     \end{psmallmatrix}  \right\}$, \\
$\mathbf{O^4_{14}}$ & $\mathcal{N}_{\pi}= \left\{ \begin{psmallmatrix}
\sigma &  0  \\
 0    &  \pm 1      \end{psmallmatrix} \mid \sigma \in \left\{  \mathrm{diag}( \pm 1, \pm 1, \pm 1) \rtimes \mathcal{S}_3  \right\} \right\}$, \\
$\mathbf{O^4_{15}}$ & $\mathcal{N}_{\pi}= \left\{ \mathrm{diag}(\pm 1, \pm 1, \pm 1, \pm 1) , \begin{psmallmatrix}
\pm 1 & 0 & 0 & 0  \\
0 & 0 & 0 & \pm 1 \\
0 & 0 & \pm 1 & 0 \\ 
0 & \pm 1 & 0 &  0      \end{psmallmatrix}  \right\}$, \\
$\mathbf{O^4_{17}}$ & $\mathcal{N}_{\pi}= \left\{ \mathrm{diag}(\pm 1, \pm 1, \pm 1, \pm 1) , \begin{psmallmatrix}
\pm 1 & 0 & 0 & 0  \\
0 & \pm 1 & 0 & 0 \\
0 & 0 & 0 & \pm 1  \\ 
0 & 0 & \pm 1 &  0      \end{psmallmatrix}  \right\}$. \\
\end{tabular}
\end{proposition}
\begin{proof}
These cases correspond to table \ref{table:1}. All of them have one fixed entry, in the sense that for that entry the action of the generators is trivial. The only  matrices that could belong to $\mathcal{N}_{\pi}$, besides $\mathrm{diag}(\pm 1, \pm 1, \pm 1, \pm 1)$, are permutation matrices of three coordinates. For each case we analyze which permutations are in $\mathcal{N}_{\pi}$. First, we consider the matrix part of the generators $\alpha$, $\beta$, and $\alpha \beta $, in order to see which permutations of the entries $-1$ in the diagonal of the matrices are possible. Second, for the cases without a standard lattice, we compute the type of matrices that preserve the lattice as in \cite{karla}. Finally, we analyze the permutation matrices $P$ preserving the corresponding translation of each generator, which means that we have to find a translation $x \in \R^4$ with  
\[
(P,x) \alpha (P,x)^{-1} = (B, y), \quad \text{where} \quad \tau(\beta)=B,  
\]
such that the translation $y$ is again a possible translation for the generator $\beta$. In most cases, we do not have such a translation.     
\end{proof}

Now, we study the $4$-dimensional closed flat manifolds with holonomy $H_{\pi}= \Z_2^3$; their Bieberbach groups are in table \ref{table:2}. 

\begin{table}
\begin{tabular}{|l|c|c|}
\hline
$\mathbf{N^4_{30}}$ &    $\pi = \langle t_1, t_2, t_3, t_4, \alpha, \beta, \gamma \rangle $       &  $\alpha  = \left(  \mathrm{diag}(-1,1,1,1) , \frac{1}{2}e_3  \right) $\\ 
      &        &  $\beta= \left(  \mathrm{diag}(-1,1,-1,1) , \frac{1}{2}e_2 - \frac{1}{2}e_3  \right)$  \\
    &    & $ \gamma = \left( \mathrm{diag}(-1,-1,1,1) , \frac{1}{2}e_4 \right)$  \\
 \hline
$\mathbf{N^4_{31}}$ &    $\pi = \langle t_1, t_2, t_3, t_4, \alpha, \beta, \gamma \rangle $       &  $\alpha  = \left( \mathrm{diag}(1,1,-1,1) , \frac{1}{2}e_2  \right) $\\ 
      &        &  $\beta= \left( \mathrm{diag}(1,-1,-1,1) , \frac{1}{2}e_1 - \frac{1}{2}e_2  \right)$  \\
    &    & $ \gamma = \left( \mathrm{diag}(-1,-1,1,1) , \frac{1}{2}e_1 + \frac{1}{2}e_4 \right)$  \\
 \hline
$\mathbf{N^4_{32}}$ &    $\pi = \langle t_1, t_2, t_3, t_4, \alpha, \beta, \gamma \rangle $       &  $\alpha  = \left( \mathrm{diag}(-1,1,1,1) , \frac{1}{2}e_3  \right) $\\ 
      &        &  $\beta= \left( \mathrm{diag}(-1,1,-1,1) , \frac{1}{2}e_2 - \frac{1}{2}e_3  \right)$  \\
    &    & $ \gamma = \left( \mathrm{diag}(-1,-1,1,1) , \frac{1}{2}e_1 + \frac{1}{2}e_4 \right)$  \\
 \hline
$\mathbf{N^4_{33}}$ &    $\pi = \langle t_1, t_2, t_3, t_4, \alpha, \beta, \gamma \rangle $       &  $\alpha  = \left( \mathrm{diag}(-1,1,1,1) , \frac{1}{2}e_3  \right) $\\ 
      &        &  $\beta= \left( \mathrm{diag}(-1,1,-1,1) , \frac{1}{2}e_2 - \frac{1}{2}e_3  \right)$  \\
    &    & $ \gamma = \left( \mathrm{diag}(-1,-1,1,1) , \frac{1}{2}e_1 + \frac{1}{2}e_3 + \frac{1}{2}e_4 \right)$  \\
 \hline
$\mathbf{N^4_{34}}$ &    $\pi = \langle t_1, t_2, t_3, t_4, \alpha, \beta, \gamma \rangle $       &  $\alpha  = \left( \mathrm{diag}(-1,1,1,1) , \frac{1}{2}e_3  \right) $\\ 
      &        &  $\beta= \left( \mathrm{diag}(-1,1,-1,1) , \frac{1}{2}e_1 + \frac{1}{2}e_2 + \frac{1}{2}e_3  \right)$  \\
    &    & $ \gamma = \left( \mathrm{diag}(-1,-1,1,1) , \frac{1}{2}e_1 +  \frac{1}{2}e_4 \right)$  \\
 \hline
$\mathbf{N^4_{35}}$ &    $\pi = \langle t_1, t_2, t_3, t_4, \alpha, \beta, \gamma \rangle $       &  $\alpha  = \left( \mathrm{diag}(-1,1,1,1) , \frac{1}{2}e_3  \right) $\\ 
      &        &  $\beta= \left( \mathrm{diag}(-1,1,-1,1) , \frac{1}{2}e_1 + \frac{1}{2}e_2 - \frac{1}{2}e_3  \right)$  \\
    &    & $ \gamma = \left( \mathrm{diag}(1,-1,-1,1) , -\frac{1}{2}e_2 +  \frac{1}{2}e_4 \right)$  \\
\hline
$\mathbf{N^4_{36}}$ &    $\pi = \langle t_1, t_2, t_3, t_4, \alpha, \beta, \gamma \rangle $       &  $\alpha  = \left( \mathrm{diag}(-1,1,1,1) , \frac{1}{2}e_3  \right) $\\ 
      &        &  $\beta= \left( \mathrm{diag}(-1,1,-1,1) , -\frac{1}{2}e_1 + \frac{1}{2}e_2 - \frac{1}{2}e_3  \right)$  \\
    &    & $ \gamma = \left( \mathrm{diag}(-1,-1,1,1) ,  \frac{1}{2}e_4 \right)$  \\
\hline
$\mathbf{N^4_{37}}$ &    $\pi = \langle t_1, t_2, t_3, t_4, \alpha, \beta, \gamma \rangle $       &  $\alpha  = \left( \mathrm{diag}(1,1,-1,1) , \frac{1}{2}e_2  \right) $\\ 
      &        &  $\beta= \left( \mathrm{diag}(1,-1,-1,1) , \frac{1}{2}e_1 - \frac{1}{2}e_2 - \frac{1}{2}e_3  \right)$  \\
    &    & $ \gamma = \left( \mathrm{diag}(-1,-1,1,1) , \frac{1}{2}e_1 + \frac{1}{2}e_3 +  \frac{1}{2}e_4 \right)$  \\
\hline
$\mathbf{N^4_{38}}$ &    $\pi = \langle t_1, t_2, t_3, t_4, \alpha, \beta, \gamma \rangle $       &  $\alpha  = \left( \mathrm{diag}(-1,1,-1,1) , \frac{1}{2}e_2  \right) $\\ 
      &        &  $\beta= \left( \mathrm{diag}(1,-1,-1,1) , -\frac{1}{2}e_1 + \frac{1}{2}e_2 - \frac{1}{2}e_3  \right)$  \\
    &    & $ \gamma = \left( \mathrm{diag}(-1,1,1,1) ,  \frac{1}{2}e_4 \right)$  \\
\hline
$\mathbf{N^4_{39}}$ &    $\pi = \langle t_1, t_2, t_3, t_4, \alpha, \beta, \gamma \rangle $       &  $\alpha  = \left( \mathrm{diag}(-1,1,-1,1) , \frac{1}{2}e_2 - \frac{1}{2}e_3  \right) $\\ 
      &        &  $\beta= \left( \mathrm{diag}(-1,-1,1,1) , \frac{1}{2}e_1 + \frac{1}{2}e_3  \right)$  \\
    &    & $ \gamma = \left( \mathrm{diag}(1,1,-1,1) , \frac{1}{2}e_2 + \frac{1}{2}e_4 \right)$  \\
\hline
$\mathbf{N^4_{45}}$ &    $\pi = \langle t_1, t_2, t_3, t_4, \alpha, \beta, \gamma \rangle $       &  $\alpha  = \left( \mathrm{diag}(1,-1,-1,-1) , -\frac{1}{2}e_1 - \frac{1}{2}e_2  \right) $\\ 
      &        &  $\beta= \left( \mathrm{diag}(-1,-1,1,1) , \frac{1}{2}e_1 + \frac{1}{2}e_4  \right)$  \\
    &    & $ \gamma = \left( \mathrm{diag}(1,-1,-1,1) , -\frac{1}{2}e_1 - \frac{1}{2}e_3 \right)$  \\
 \hline
$\mathbf{N^4_{46}}$ &    $\pi = \langle t_1, t_2, t_3, t_4, \alpha, \beta, \gamma \rangle $       &  $\alpha  = \left( \mathrm{diag}(-1,1,-1,-1) , \frac{1}{2}e_2  \right) $\\ 
      &        &  $\beta= \left( \mathrm{diag}(-1,-1,1,1) , \frac{1}{2}e_1 + \frac{1}{2}e_4  \right)$  \\
    &    & $ \gamma = \left( \mathrm{diag}(-1,1,-1,1) , \frac{1}{2}e_2 - \frac{1}{2}e_3 \right)$  \\
 \hline
\end{tabular}
\caption{Bieberbach groups with holonomy $\Z_2^3$. See pages 60-65 and 68 in \cite{lambert}.}
\label{table:2}
\end{table}


\begin{proposition} \label{case 2}
The matrix part of the normalizer of $\pi$ in $\A(4)$ for the $4$-dimen\-sional closed flat manifolds with holonomy $\mathbb{Z}_2^3$ is as follows:
\[
\begin{array}{cc}
\mathbf{N^4_{30}}, \mathbf{N^4_{31}}, \mathbf{N^4_{32}}, \mathbf{N^4_{33}}, \mathbf{N^4_{34}}, \mathbf{N^4_{35}}, \mathbf{N^4_{36}}, \mathbf{N^4_{37}}, \mathbf{N^4_{38}}, \mathbf{N^4_{39}}  & \mathcal{N}_{\pi}= \left\{ \mathrm{diag}(\pm 1, \pm 1, \pm 1, \pm 1)   \right\},
\end{array}
\]
\[
\begin{array}{cl}
\mathbf{N^4_{45}} & \mathcal{N}_{\pi}= \left\{ \mathrm{diag}(\pm 1, \pm 1, \pm 1, \pm 1), \begin{psmallmatrix}
\pm 1 & 0 & 0 & 0  \\
0 & 0 & \pm 1 & 0 \\
0 & \pm 1 & 0 & 0  \\ 
0 & 0 & 0 &  \pm 1      \end{psmallmatrix}    \right\}, \\
\mathbf{N^4_{46}}  & \mathcal{N}_{\pi}= \left\{ \mathrm{diag}(\pm 1, \pm 1, \pm 1, \pm 1)   \right\}.
\end{array}
\]
\end{proposition}

\begin{proof}
These cases correspond to table \ref{table:2}. For the manifolds with one fixed entry, the matrix group is 
\[
\begin{split}
H_{\pi} = \left\{ \Id, \mathrm{diag}(-1,1,1,1), \mathrm{diag}(1,-1,1,1), \mathrm{diag}(1,1,-1,1), \quad  \right. \\ 
\left. \mathrm{diag}(-1,-1,1,1), \mathrm{diag}(-1,1,-1,1), \mathrm{diag}(1,-1,-1,1), \mathrm{diag}(-1,-1,-1,1) \right\}.
\end{split} 
\]
The only possibilities for permutation are between the ones with only one $-1$ in the diagonal or between the ones with two $-1$ in the diagonal, by lemma \ref{notinv}. Moreover, if we do not have the permutations for one $-1$ in $\mathcal{N}_{\pi}$ then it is not possible for the other permutations of two $-1$ to be in $\mathcal{N}_{\pi}$ and vice versa. The permutations do not preserve the translations of the generators, then the matrix part of the normalizer is $\mathcal{N}_{\pi}= \left\{ \mathrm{diag}(\pm 1, \pm 1, \pm 1, \pm 1)   \right\}$  in all cases.   
 
Now, we consider the manifolds without fixed entries. In this case, the matrix group is 
\begin{multline*}
H_{\pi} = \left\{ \Id, \mathrm{diag}(-1,-1,1,1), \mathrm{diag}(1,-1,-1,1), \mathrm{diag}(-1,1,-1,1) \quad  \right. \\ 
\left. \mathrm{diag}(1,-1,-1,-1), \mathrm{diag}(-1,1,-1,-1), \mathrm{diag}(-1,-1,1,-1), \mathrm{diag}(1,1,1,-1) \right\} .
\end{multline*} 
Again, from lemma \ref{notinv} we know that the only possibilities for permutation are between the ones with only two $-1$ in the diagonal or between the ones with three $-1$ in the diagonal. The situation is similar to the case of one fixed entry. For this case, we have that one flat manifold, $N^4_{45}$, admits one permutation.     
\end{proof}


The manifolds with Bieberbach groups in table \ref{table:3}, have holonomy $H_{\pi}=\mathrm{A}_4$ or $\mathrm{A}_4 \times \Z_2$. 

\begin{table}
\begin{tabular}{|l|c|c|}
\hline
$\mathbf{O^4_{26}}$ &    $\pi = \langle t_1, t_2, t_3, t_4, \alpha, \beta, \gamma \rangle $       &  $\alpha  = \left( \mathrm{diag}(1,1,-1,-1) , -\frac{1}{2}e_2 - \frac{1}{2}e_4  \right) $\\ 
      &        &  $\beta= \left( \mathrm{diag}(1,-1,-1,1) , \frac{1}{2}e_3 + \frac{1}{2}e_4  \right)$  \\
    &    & $ \gamma = \left( \begin{psmallmatrix}
1 & 0 & 0 & 0 \\
0 & 0 & 0 & 1  \\
0 & 1 & 0 & 0 \\
0 & 0 & 1 & 0  \\      \end{psmallmatrix} , \frac{1}{3}e_1 + \frac{1}{2}e_2 - \frac{1}{2}e_4 \right)$  \\
 \hline
$\mathbf{O^4_{27}}$ &    $\pi = \langle t_1, s_1, s_2, s_3, \alpha, \beta, \gamma \rangle $       &  $\alpha  = \left(  \mathrm{diag}(1,-1,1,-1) , \frac{1}{2}e_1 + \frac{1}{2}e_3 + \frac{1}{2}e_4  \right) $\\ 
      &  $s_1= (\Id , \frac{1}{2}(-e_1 +e_2-e_3+e_4) )$   &  $\beta= \left( \mathrm{diag}(1,-1,-1,1) , -\frac{1}{2}e_2 + \frac{1}{2}e_4  \right)$  \\
    & $s_2= (\Id , \frac{1}{2}(-e_1 +e_2+e_3-e_4) )$  & $ \gamma = \left( \begin{psmallmatrix}
1 & 0 & 0 & 0 \\
0 & 0 & -1 & 0  \\
0 & 0 & 0 & -1 \\
0 & 1 & 0 & 0  \\      \end{psmallmatrix} , \frac{1}{6}e_1 + \frac{1}{2}e_3  \right)$  \\
  &  $s_3= (\Id , \frac{1}{2}(e_1 - e_2+e_3+e_4) )$  &     \\
 \hline
$\mathbf{N^4_{43}}$ &    $\pi = \langle t_1, t_2, t_3, t_4, \alpha, \beta, \gamma \rangle $       &  $\alpha  = \left(  \mathrm{diag}(1,1,-1,-1) , \frac{1}{2}e_2 + \frac{1}{2}e_3  \right) $\\ 
      &        &  $\beta= \left(  \mathrm{diag}(1,-1,-1,1) , \frac{1}{2}e_2 - \frac{1}{2}e_4  \right)$  \\
    &    & $ \gamma = \left(  \begin{psmallmatrix}
1 & 0 & 0 & 0 \\
0 & 0 & 0 & -1  \\
0 & 1 & 0 & 0 \\
0 & 0 & 1 & 0  \\      \end{psmallmatrix} , \frac{1}{6}e_1  \right)$  \\
 \hline
\end{tabular}
\caption{Bieberbach groups with holonomy $\mathrm{A}_4$ or $\mathrm{A}_4 \times \Z_2$. See pages 47, 48 and 67 in \cite{lambert}.}
\label{table:3}
\end{table}

Since the representation of $\mathrm{A}_4$ in $\R^3$ is irreducible, we have two isotypic components $\R^4 = V_1 \times V_2$, where $V_2 = \R^3$ and End$_{H_{\pi}}(V_2) \cong \{ \lambda\, \Id \mid \lambda \in \R \} \cong \R$. This means that we have only two different $1$-dimensional isotypic components of real type, and the action of $\mathcal{N}_{\pi}$ reduces to an action on $\R^2$.

\begin{proposition} \label{case 3}
The matrix part of the normalizer of $\pi$ in $\A(4)$ for the $4$-dimen\-sional closed flat manifolds with holonomy $\mathrm{A}_4$ and $\mathrm{A}_4 \times \mathbb{Z}_2 $ is as follows:
\[
\begin{array}{cc}
\mathbf{O^4_{26}} & \mathcal{N}_{\pi}= \left\{ \begin{psmallmatrix}
\pm 1 & 0   \\
0 & \pm \Id   \\   \end{psmallmatrix}  , \begin{psmallmatrix}
\pm 1 & 0   \\
0 & \pm \sigma_1   \\   \end{psmallmatrix}, \begin{psmallmatrix}
\pm 1 & 0   \\
0 & \pm \sigma_2   \\   \end{psmallmatrix} \mid \sigma_1= \begin{psmallmatrix}
0 & 1 & 0   \\
0 & 0 & 1   \\ 
1 & 0 & 0  \end{psmallmatrix}, \sigma_2 = \begin{psmallmatrix}
0 & 0 & 1   \\
1 & 0 & 0   \\ 
0 & 1 & 0  \end{psmallmatrix}    \right\} , \\
\mathbf{O^4_{27}}  & \mathcal{N}_{\pi}= \left\{ \begin{psmallmatrix}
\pm 1 & 0   \\
0 & \pm \Id   \\   \end{psmallmatrix}   \right\} , \\
\mathbf{N^4_{43}} & \mathcal{N}_{\pi}= \left\{ \begin{psmallmatrix}
\pm 1 & 0   \\
0 & \pm \Id   \\   \end{psmallmatrix}  , \begin{psmallmatrix}
\pm 1 & 0   \\
0 & \pm \sigma_1   \\   \end{psmallmatrix}, \begin{psmallmatrix}
\pm 1 & 0   \\
0 & \pm \sigma_2   \\   \end{psmallmatrix} \mid \sigma_1= \begin{psmallmatrix}
0 & 1 & 0   \\
0 & 0 & 1   \\ 
-1 & 0 & 0  \end{psmallmatrix}, \sigma_2 = \begin{psmallmatrix}
0 & 0 & 1   \\
-1 & 0 & 0   \\ 
0 & -1 & 0  \end{psmallmatrix}    \right\} . \\
\end{array}
\]
\end{proposition}

\begin{proof}
These cases correspond to table \ref{table:3}. We have the option of permutations in  $\mathcal{N}_{\pi}$, which send the generator $\alpha$ either to the generator $\beta$ or to the generator $\alpha \beta$.   
\end{proof}

\subsection{One $2$-dimensional isotypic component of real type and two $1$-dimen\-sional components.}
The manifolds of these cases have holonomy $H_{\pi}= \Z_2^2$ and their Bieberbach groups are in table \ref{table:4}.

\begin{table}
\begin{tabular}{|l|c|c|}
\hline
$\mathbf{N^4_{3}}$ &    $\pi = \langle t_1, t_2, t_3, t_4, \alpha, \beta \rangle $       &  $\alpha  = \left( \mathrm{diag}(1,1,1,-1) , \frac{1}{2}e_2 \right) $\\ 
      &        &  $\beta= \left( \mathrm{diag}(1,1,-1,-1) , \frac{1}{2}e_1 + \frac{1}{2}e_2 \right)$  \\
 \hline
$\mathbf{N^4_{4}}$ &    $\pi = \langle t_1, t_2, t_3, t_4, \alpha, \beta \rangle $       &  $\alpha  = \left( \mathrm{diag}(1,1,1,-1) , \frac{1}{2}e_2 \right) $\\ 
      &        &  $\beta= \left( \mathrm{diag}(1,1,-1,-1) , \frac{1}{2}e_1 + \frac{1}{2}e_2 + \frac{1}{2}e_4 \right)$  \\
 \hline
$\mathbf{N^4_{5}}$ &    $\pi = \langle t_1, t_2, t_3, t_4, \alpha, \beta \rangle $       &  $\alpha  = \left(  \mathrm{diag}(1,1,-1,1) , \frac{1}{2}e_1 + \frac{1}{2}e_4 \right) $\\ 
      &        &  $\beta= \left( \mathrm{diag}(1,1,-1,-1) , \frac{1}{2}e_1 + \frac{1}{2}e_2 + \frac{1}{2}e_3 - \frac{1}{2}e_4 \right)$  \\
 \hline
$\mathbf{N^4_{6}}$ &    $\pi = \langle t_1, t_2, t_3, s_1, \alpha, \beta \rangle $       &  $\alpha  = \left( \mathrm{diag}(1,1,1,-1) , \frac{1}{2}e_2  \right) $\\ 
      &  $s_1=(\Id, \frac{1}{2}(e_3 + e_4))$  &  $\beta= \left(  \mathrm{diag}(1,1,-1,-1) , \frac{1}{2}e_1 + \frac{1}{2}e_2  \right)$  \\
 \hline
$\mathbf{N^4_{7}}$ &    $\pi = \langle t_1, t_2, s_1, s_2, \alpha, \beta \rangle $       &  $\alpha  = \left( \mathrm{diag}(1,1,1,-1) , \frac{1}{4}e_1 + \frac{1}{4}e_3  \right) $\\ 
      &  $s_1=(\Id, \frac{1}{2}(e_1 + e_3))$  &  $\beta= \left(  \mathrm{diag}(1,1,-1,-1) , \frac{1}{2}e_1 + \frac{1}{2}e_2 - \frac{1}{4}e_3 + \frac{1}{4}e_4  \right)$  \\
    & $s_2=(\Id, \frac{1}{2}(e_1 + e_4))$  &   \\
 \hline
$\mathbf{N^4_{8}}$ &    $\pi = \langle t_1, t_2, t_3, t_4, \alpha, \beta \rangle $       &  $\alpha  = \left(  \mathrm{diag}(1,1,1,-1) , \frac{1}{2}e_3  \right) $\\ 
      &     &  $\beta= \left(  \mathrm{diag}(1,1,-1,-1) , \frac{1}{2}e_2 - \frac{1}{2}e_3  \right)$  \\
\hline
$\mathbf{N^4_{9}}$ &    $\pi = \langle t_1, t_2, t_3, t_4, \alpha, \beta \rangle $       &  $\alpha  = \left(  \mathrm{diag}(1,1,1,-1) , \frac{1}{2}e_3  \right) $\\ 
      &     &  $\beta= \left( \mathrm{diag}(1,1,-1,-1) , \frac{1}{2}e_2 - \frac{1}{2}e_3 + \frac{1}{2}e_4  \right)$  \\
 \hline
$\mathbf{N^4_{10}}$ &    $\pi = \langle t_1, t_2, s_1, t_4, \alpha, \beta \rangle $       &  $\alpha  = \left( \mathrm{diag}(1,1,1,-1) , \frac{1}{2}e_1  \right) $\\ 
      &  $s_1=(\Id, \frac{1}{2}(e_1 + e_3))$  &  $\beta= \left(  \mathrm{diag}(1,1,-1,-1) , \frac{1}{2}e_1 + \frac{1}{2}e_2  \right)$  \\
 \hline
$\mathbf{N^4_{11}}$ &    $\pi = \langle t_1, t_2, s_1, t_4, \alpha, \beta \rangle $       &  $\alpha  = \left(  \mathrm{diag}(1,1,-1,1) , \frac{1}{2}e_4  \right) $\\ 
      &  $s_1=(\Id, \frac{1}{2}(e_1 + e_3))$  &  $\beta= \left(  \mathrm{diag}(1,1,-1,-1) , \frac{1}{2}e_2 + \frac{1}{2}e_4  \right)$  \\
 \hline
$\mathbf{N^4_{12}}$ &    $\pi = \langle t_1, t_2, s_1, t_4, \alpha, \beta \rangle $       &  $\alpha  = \left( \mathrm{diag}(1,1,1,-1) , \frac{1}{2}e_1  \right) $\\ 
      &  $s_1=(\Id, \frac{1}{2}(e_1 + e_3))$  &  $\beta= \left(  \mathrm{diag}(1,1,-1,-1) , \frac{1}{2}e_1 + \frac{1}{2}e_2 + \frac{1}{2}e_4  \right)$  \\
 \hline
$\mathbf{N^4_{13}}$ &    $\pi = \langle t_1, t_2, t_3, s_1, \alpha, \beta \rangle $       &  $\alpha  = \left( \mathrm{diag}(1,1,-1,1) , -\frac{1}{2}e_1  \right) $\\ 
      &  $s_1=(\Id, \frac{1}{2}(-e_1 + e_3 + e_4))$  &  $\beta= \left(  \mathrm{diag}(1,1,-1,-1) , \frac{1}{2}e_1 + \frac{1}{2}e_2  \right)$  \\
 \hline
$\mathbf{N^4_{22}}$ &    $\pi = \langle t_1, t_2, t_3, t_4, \alpha, \beta \rangle $       &  $\alpha  = \left( \mathrm{diag}(-1,-1,1,1) , \frac{1}{2}e_3 \right) $\\ 
      &        &  $\beta= \left( \mathrm{diag}(-1,-1,-1,1) , \frac{1}{2}e_3 + \frac{1}{2}e_4 \right)$  \\
 \hline
$\mathbf{N^4_{23}}$ &    $\pi = \langle t_1, t_2, t_3, t_4, \alpha, \beta \rangle $       &  $\alpha  = \left( \mathrm{diag}(-1,-1,1,1) , -\frac{1}{2}e_2 + \frac{1}{2}e_3 \right) $\\ 
      &        &  $\beta= \left( \mathrm{diag}(-1,-1,-1,1) , \frac{1}{2}e_3 + \frac{1}{2}e_4 \right)$  \\
 \hline
$\mathbf{N^4_{24}}$ &    $\pi = \langle t_1, t_2, t_3, t_4, \alpha, \beta \rangle $       &  $\alpha  = \left( \mathrm{diag}(1,1,-1,1) , \frac{1}{2}e_2  \right) $\\ 
      &        &  $\beta= \left( \mathrm{diag}(-1,-1,1,1) , \frac{1}{2}e_2 + \frac{1}{2}e_4 \right)$  \\
 \hline
$\mathbf{N^4_{25}}$ &    $\pi = \langle t_1, t_2, t_3, t_4, \alpha, \beta \rangle $       &  $\alpha  = \left( \mathrm{diag}(1,1,-1,1) , \frac{1}{2}e_2  \right) $\\ 
      &        &  $\beta= \left( \mathrm{diag}(-1,-1,1,1) , \frac{1}{2}e_2 - \frac{1}{2}e_3 + \frac{1}{2}e_4 \right)$  \\
 \hline
$\mathbf{N^4_{26}}$ &    $\pi = \langle t_1, t_2, s_1, t_4, \alpha, \beta \rangle $       &  $\alpha  = \left( \mathrm{diag}(1,1,-1,1) , \frac{1}{2}e_2  \right) $\\ 
      &  $s_1=(\Id , \frac{1}{2}(e_1 + e_3))$  &  $\beta= \left( \mathrm{diag}(-1,-1,1,1) , \frac{1}{2}e_2 + \frac{1}{2}e_4 \right)$  \\
\hline
$\mathbf{N^4_{44}}$ &    $\pi = \langle t_1, t_2, t_3, t_4, \alpha, \beta \rangle $       &  $\alpha  = \left( \mathrm{diag}(-1,-1,-1,1) , \frac{1}{2}e_2 + \frac{1}{2}e_4 \right) $\\ 
      &        &  $\beta= \left( \mathrm{diag}(1,1,-1,-1) , \frac{1}{2}e_2 - \frac{1}{2}e_3 - \frac{1}{2}e_4 \right)$  \\
 \hline
\end{tabular}
\caption{Bieberbach groups with holonomy $\Z_2^2$, one $2$-dimensional isotypic component of real type and two $1$-dimensional components. See pages 50-54, 57-59, and 67 in \cite{lambert}.}
\label{table:4}
\end{table}

Let us consider a matrix of the form 
$ X= \begin{psmallmatrix}
a & b \\ c & d \end{psmallmatrix} $ and denote 
\begin{align*}
\Gamma_0(4) & = \{ X \in \G(2, \Z) \mid b \equiv 0 \text{ mod }2 \text{ and } c\equiv 0 \text{ mod 4 } \} \\
\Gamma_2(4) & = \{ X \in \G(2, \Z) \mid b \equiv 0 \text{ mod }2 \text{ and } c\equiv 2 \text{ mod 4 } \}.
\end{align*}

\begin{proposition} \label{case 4}
The matrix part of the normalizer of $\pi$ in $\A(4)$ for the $4$-dimen\-sional closed flat manifolds with holonomy $\mathbb{Z}_2 ^2$, with one $2$-dimensional isotypic component of real type and two $1$-dimensional isotypic components of real type, is as follows:
\[
\renewcommand{\arraystretch}{1.8}
\begin{array}{cc} 
\mathbf{N^4_{22}}  &  \mathcal{N}_{\pi}= \left\{ \begin{psmallmatrix}
X  & 0 & 0   \\
0 & \pm 1 & 0 \\
0 &  0  &  \pm 1  \\   \end{psmallmatrix} \mid  X \in \G(2,\Z)  \right\}  , \\
\mathbf{N^4_{8}}, \mathbf{N^4_{9}}, \mathbf{N^4_{23}}, \mathbf{N^4_{24}}, \mathbf{N^4_{25}}, \mathbf{N^4_{26}}    & \mathcal{N}_{\pi}= \left\{ \begin{psmallmatrix}
X  & 0 & 0   \\
0 & \pm 1 & 0 \\
0 &  0  &  \pm 1  \\   \end{psmallmatrix} \mid X \in \Gamma_0(2)^{t}  \right\} , \\
\mathbf{N^4_{3}}, \mathbf{N^4_{5}}, \mathbf{N^4_{6}} & \begin{array}{c}
\mathcal{N}_{\pi}= \left\langle \begin{psmallmatrix}
X_1  & 0 & 0   \\
0 & \pm 1 & 0 \\
0 &  0  &  \pm 1  \\   \end{psmallmatrix}, \begin{psmallmatrix}
X_2  & 0 & 0   \\
0 &  0  &   1  \\
0 &  1 & 0 \\   \end{psmallmatrix} \mid X_1 \in \Gamma(2), \right. \\
\left. \; X_2 \in \Gamma(2)\cdot \begin{psmallmatrix}
0  &   1  \\
1 & 0 \\   \end{psmallmatrix}  \right\rangle , \end{array}  \\ 
\mathbf{N^4_{4}},  \mathbf{N^4_{10}}, \mathbf{N^4_{11}}, \mathbf{N^4_{12}}, \mathbf{N^4_{13}} &  \mathcal{N}_{\pi}= \left\{ \begin{psmallmatrix}
X  & 0 & 0   \\
0 & \pm 1 & 0 \\
0 &  0  &  \pm 1  \\   \end{psmallmatrix} \mid X \in \Gamma(2)  \right\} , \\
 \mathbf{N^4_{7}}  & \begin{array}{c}
\mathcal{N}_{\pi}= \left\langle \begin{psmallmatrix}
X_1  & 0 & 0   \\
0 & \pm 1 & 0 \\
0 &  0  &  \pm 1  \\   \end{psmallmatrix}, \begin{psmallmatrix}
X_2  & 0 & 0   \\
0 &  0  &   1  \\
0 &  1 & 0 \\   \end{psmallmatrix} \mid X_1 \in \Gamma_0(4), \right. \\
\left. \; X_2 \in \Gamma_2(4) \right\rangle , \end{array}  \\ 
\mathbf{N^4_{44}}  & \mathcal{N}_{\pi}= \left\langle \begin{psmallmatrix}
X_1  & 0 & 0   \\
0 & \pm 1 & 0 \\
0 &  0  &  \pm 1  \\   \end{psmallmatrix}, \begin{psmallmatrix}
\Id & 0 & 0   \\
0  &  0 & 1 \\ 
0 &   1  & 0  \end{psmallmatrix} \mid X_1 \in \Gamma_0(2)^{t} \right\rangle . \\
\end{array}
\]

\end{proposition}

\begin{proof}
These cases correspond to table \ref{table:4}. We classify them according to the number of entries with trivial action, i.e., the number of fixed entries.
 
First, we study the cases with one fixed entry. These are the manifolds $\mathbf{N^4_{i}}$ with $i \in \{ 22,\dots,26\}$. We have one fixed entry and the number of $-1$ appearing in the diagonal of the generators are different, then there are no matrices that permute the generators in the normalizer. The matrices fixing each generator have the form 
\begin{equation} \label{form}
\begin{psmallmatrix}
X  & 0 & 0   \\
0 & \pm 1 & 0 \\
0 &  0  &  \pm 1  \\   \end{psmallmatrix} \quad \text{with} \quad X \in \G(2,\Z),
\end{equation}
because we have a $2$-dimensional isotypic component of $\R$ type. We look for the matrices $X$ that preserve the corresponding lattices of the generators, following the method developed in \cite{karla}. For $\mathbf{N^4_{22}}$, we get that there is no restriction for $X \in \G(2,\Z)$, but for $\mathbf{N^4_{23}}$, $\mathbf{N^4_{24}}$, $\mathbf{N^4_{25}}$, and $\mathbf{N^4_{26}}$ we get that $ X\in \Gamma_0(2)^{t} $.    

The second case is the one with two fixed entries. These are the manifolds $\mathbf{N^4_{i}}$ with $i \in \{ 3,\dots,13\}$. The only possible permutation is between $\mathrm{diag}(1,1,1,-1)$ and $\mathrm{diag}(1,1,-1,1)$. The matrices fixing the generators have the same form as in \eqref{form}, and the ones that switch them have the form   
\[
\begin{psmallmatrix}
X  & 0 & 0   \\
0 &  0  &  \pm 1  \\
0 & \pm 1 & 0 \\   \end{psmallmatrix} \quad \text{with} \quad X \in \G(2,\Z).
\]  
Depending on the translation part of the generators we obtain the different results. 

The last case is $\mathbf{N^4_{44}}$, which does not have fixed entries. It admits the permutation $P$ between $\mathrm{diag}(-1,-1,-1,1)$ and $\mathrm{diag}(-1,-1,1,-1)$, where $(P,\frac{1}{4}) \in \mathrm{N}_{\A(n)}(\pi) $.  
\end{proof}

\section{Manifolds with one isotypic component of complex type}\label{sec:isotypic1complex}

The remaining $4$-dimensional closed flat manifolds with two generators in their holonomy have one $1$-dimensional isotypic component of $\C$ type and two $1$-dimensional isotypic components of $\R$ type. In order to enumerate them, we consider the following notation. Denote the rotation matrix by an angle $\theta \in [0, 2\pi]$ as
$$ R(\theta)= \left( \begin{array}{cc}
\cos \theta & -\sin \theta    \\ 
\sin \theta & \cos \theta      \end{array} \right),
$$
and denote the reflection as $E_0 = \mathrm{diag}(1,-1)$. 
We classify the manifolds according to their holonomy.


\begin{table}  
\begin{tabular}{|l|c|c|}
\hline
$\mathbf{O^4_{18}}$ &    $\pi = \langle t_1, t_2, t_3, s_1, \alpha, \beta \rangle $       &  $\alpha  = \left( A= \begin{psmallmatrix}
\Id & 0  \\
 0 & R(\frac{4\uppi}{3})
  \end{psmallmatrix} , \frac{1}{3}e_2 \right) $\\ 
      &  $s_1=(\Id, -A(e_3))$ &  $\beta= \left( B= \begin{psmallmatrix}
E_0 & 0  \\
 0 & R(\frac{4\uppi}{3})E_0
  \end{psmallmatrix} , \frac{1}{2}e_1 + \frac{1}{3}e_2 \right)$  \\
 \hline
$\mathbf{O^4_{19}}$ &    $\pi = \langle t_1, t_2, t_3, s_1, \alpha, \beta \rangle $       &  $\alpha  = \left( A= \begin{psmallmatrix}
\Id & 0  \\
 0 & R(\frac{4\uppi}{3})
  \end{psmallmatrix} , \frac{1}{3}e_1 \right) $\\ 
      &  $s_1=(\Id, -A(e_3))$ &  $\beta= \left( B= \begin{psmallmatrix}
-E_0 & 0  \\
 0 & R(\frac{4\uppi}{3})(-E_0)
  \end{psmallmatrix} , \frac{1}{3}e_1 + \frac{1}{2}e_2 \right)$  \\
 \hline
 $\mathbf{O^4_{20}}$ &    $\pi = \langle t_1, s_1, s_2, s_3, \alpha, \beta \rangle $       &  $\alpha  = \left( A= \begin{psmallmatrix}
\Id & 0  \\
 0 & R(\frac{4\uppi}{3})
  \end{psmallmatrix} , \frac{1}{3}e_1 + \frac{1}{2}e_3 - \frac{1}{2\sqrt{3}}e_4 \right) $\\ 
      &  $s_1=(\Id, v=\frac{1}{3}e_2 + \frac{2}{3}e_3)$ &  $\beta= \left( B= \begin{psmallmatrix}
-E_0 & 0  \\
 0 & E_0
  \end{psmallmatrix} , \frac{1}{3}e_1 + \frac{1}{6}e_2 - \frac{2}{3}e_3 - \frac{1}{\sqrt{3}}e_4\right)$  \\
    &  $s_2 =(\Id , A^2(v))$  &     \\
    &  $s_3 =(\Id , A(v))$   &    \\
 \hline
\end{tabular}
\caption{Bieberbach groups with holonomy $D_3$. See page 45 in \cite{lambert}.}
\label{table:5}
\end{table}

\begin{remark}
We conjugate the case of $O^4_{20}$, in table \ref{table:5}, by the affine transformation 
\[
\left( P= \begin{psmallmatrix}
1 & 0 & 0 & 0  \\
0 & \frac{1}{3} & \frac{1}{3} & \frac{1}{3} \\
0 & \frac{2}{3} & -\frac{1}{3} & -\frac{1}{3} \\
0 & 0      & \frac{1}{\sqrt{3}} & -\frac{1}{\sqrt{3}}
  \end{psmallmatrix},0 \right).   
\]
\end{remark}

\begin{table}
\noindent \begin{tabular}{|l|c|c|}
\hline
$\mathbf{O^4_{21}}$ &    $\pi = \langle t_1, t_2, t_3, t_4, \alpha, \beta \rangle $       &  $\alpha  = \left( A= \begin{psmallmatrix}
\Id & 0  \\
 0 & R(\frac{3\uppi}{2})
  \end{psmallmatrix} , \frac{1}{4}e_1 \right) $\\ 
      &        &  $\beta= \left( B= \begin{psmallmatrix}
-E_0 & 0  \\
 0 & R(\frac{3\uppi}{2})(-E_0)
  \end{psmallmatrix} ,  \frac{1}{2}e_2 \right)$  \\
 \hline
$\mathbf{O^4_{22}}$ &    $\pi = \langle t_1, t_2, t_3, t_4, \alpha, \beta \rangle $       &  $\alpha  = \left( A= \begin{psmallmatrix}
\Id & 0  \\
 0 & R(\frac{3\uppi}{2})
  \end{psmallmatrix} , \frac{1}{4}e_1 + \frac{1}{2}e_3 \right) $\\ 
      &        &  $\beta= \left( B= \begin{psmallmatrix}
-E_0 & 0  \\
 0 & R(\frac{3\uppi}{2})(-E_0)
  \end{psmallmatrix} ,  \frac{1}{2}e_2 \right)$  \\
 \hline
$\mathbf{O^4_{23}}$ &    $\pi = \langle t_1, t_2, t_3, t_4, \alpha, \beta \rangle $       &  $\alpha  = \left( A= \begin{psmallmatrix}
\Id & 0  \\
 0 & R(\frac{3\uppi}{2})
  \end{psmallmatrix} , \frac{1}{4}e_1 + \frac{1}{2}e_2 + \frac{1}{2}e_3 \right) $\\ 
      &        &  $\beta= \left( B= \begin{psmallmatrix}
-E_0 & 0  \\
 0 & R(\frac{3\uppi}{2})(-E_0)
  \end{psmallmatrix} ,  \frac{1}{2}e_2 \right)$  \\
 \hline 
$\mathbf{O^4_{24}}$ &    $\pi = \langle s_1, s_2, s_3, t_4, \alpha, \beta \rangle $       &  $\alpha  = \left( A= \begin{psmallmatrix}
R(\frac{\uppi}{2}) & 0  \\
 0 & \Id
  \end{psmallmatrix} , \frac{1}{4}(e_1 + e_2 + e_3 + e_4) \right) $\\ 
      & $s_1=(\Id, \frac{1}{2}(-e_1-e_2+e_3))$   &  $\beta= \left( B= \begin{psmallmatrix}
E_0 & 0  \\
 0 & E_0
  \end{psmallmatrix} ,  \frac{1}{2}e_3 \right)$  \\
     &  $s_2=(\Id, \frac{1}{2}(e_1-e_2 - e_3))$  &       \\
     & $s_3=(\Id, \frac{1}{2}(e_1 + e_2 +e_3))$   &       \\
 \hline 
$\mathbf{N^4_{40}}$ &    $\pi = \langle t_1, t_2, s_1, s_2, \alpha, \beta \rangle $       &  $\alpha  = \left( A= \begin{psmallmatrix}
E_0 & 0  \\
 0 & R(\frac{3\uppi}{2})
  \end{psmallmatrix} , \frac{1}{4}(e_1 + e_2 + e_3 - e_4) \right) $\\ 
      &  $s_1=(\Id, \frac{1}{2}(e_1+e_2 +e_3 -e_4))$     &  $\beta= \left( B= \begin{psmallmatrix}
E_0 & 0  \\
 0 & E_0
  \end{psmallmatrix} ,  \frac{1}{2}(e_2 - e_3 -e_4) \right)$  \\
     &  $s_2=(\Id, \frac{1}{2}(e_1+e_2 + e_3 +e_4))$  &        \\
 \hline 
$\mathbf{N^4_{41}}$ &    $\pi = \langle t_1, t_2, t_3, t_4, \alpha, \beta \rangle $       &  $\alpha  = \left( A= \begin{psmallmatrix}
0  &  1  & 0  \\
1  &  0  & 0  \\
0  &  0 & R(\frac{3\uppi}{2})
  \end{psmallmatrix} , \frac{1}{2}(e_2 + e_3) \right) $\\ 
      &       &  $\beta= \left( B= \begin{psmallmatrix}
0  &  1  & 0  \\
1  &  0  & 0  \\
0  &  0 & -E_0
  \end{psmallmatrix} ,  \frac{1}{2}(e_4) \right)$  \\
 \hline 
$\mathbf{N^4_{47}}$ &    $\pi = \langle t_1, t_2, t_3, t_4, \alpha, \beta \rangle $       &  $\alpha  = \left( A= \begin{psmallmatrix}
-E_0 & 0  \\
 0 & R(\frac{3\uppi}{2})
  \end{psmallmatrix} , \frac{1}{4}e_2 + \frac{1}{2}e_4 \right) $\\ 
      &        &  $\beta= \left( B= \begin{psmallmatrix}
E_0 & 0  \\
 0 & R(\frac{3\uppi}{2})E_0
  \end{psmallmatrix} ,  \frac{1}{2}e_1 \right)$  \\
 \hline 
\end{tabular}
\caption{Bieberbach groups with holonomy $D_4$. See pages 46, 47, 65, 66, and 69 in \cite{lambert}.}
\label{table:6}
\end{table}


\begin{table}
\begin{tabular}{|l|c|c|}
\hline
$\mathbf{O^4_{25}}$ &    $\pi = \langle t_1, t_2, t_3, s_1, \alpha, \beta \rangle $       &  $\alpha  = \left( A= \begin{psmallmatrix}
\Id & 0  \\
 0 & R(\frac{\uppi}{3})
  \end{psmallmatrix} , \frac{1}{6}e_2 \right) $\\ 
      &  $s_1 = ( \Id , A(e_3) )$  &  $\beta= \left( B= \begin{psmallmatrix}
E_0 & 0  \\
 0 & E_0
  \end{psmallmatrix} , \frac{1}{2}(e_1 - e_2) \right)$  \\
 \hline
\end{tabular}
\caption{Bieberbach groups with holonomy $D_6$. See page 47 in \cite{lambert}.}
\label{table:7}
\end{table}

\begin{table}
\begin{tabular}{|l|c|c|}
\hline
$\mathbf{N^4_{27}}$ &    $\pi = \langle t_1, t_2, t_3, t_4, \alpha, \beta \rangle $       &  $\alpha  = \left( A= \begin{psmallmatrix}
\Id & 0  \\
 0 & R(\frac{\uppi}{2})
  \end{psmallmatrix} , \frac{1}{4}e_2 - \frac{1}{2}(e_3 + e_4) \right) $\\ 
      &        &  $\beta= \left( B= \begin{psmallmatrix}
-E_0 & 0  \\
 0 & \Id
  \end{psmallmatrix} ,  \frac{1}{2}(e_3 + e_4) \right)$  \\
 \hline
$\mathbf{N^4_{28}}$ &    $\pi = \langle t_1, t_2, t_3, t_4, \alpha, \beta \rangle $       &  $\alpha  = \left( A= \begin{psmallmatrix}
\Id & 0  \\
 0 & R(\frac{\uppi}{2})
  \end{psmallmatrix} , -\frac{1}{2}e_1 + \frac{1}{4}e_2 + \frac{1}{2}(e_3 + e_4) \right) $\\ 
      &        &  $\beta= \left( B= \begin{psmallmatrix}
-E_0 & 0  \\
 0 & \Id
  \end{psmallmatrix} ,  \frac{1}{2}(e_1 + e_3 + e_4) \right)$  \\
 \hline
$\mathbf{N^4_{29}}$ &    $\pi = \langle s_1, s_2, s_3, t_4, \alpha, \beta \rangle $       &  $\alpha  = \left( A= \begin{psmallmatrix}
R(\frac{3\uppi}{2}) & 0  \\
 0 & \Id
  \end{psmallmatrix} , \frac{1}{4}(-e_1 - e_2 - e_3 + e_4) \right) $\\ 
      &  $s_1=(\Id, \frac{1}{2}(-e_1-e_2+e_3))$   &  $\beta= \left( B= \begin{psmallmatrix}
\Id & 0  \\
 0 & -E_0
  \end{psmallmatrix} ,  -\frac{1}{2}e_2  \right)$  \\
    & $s_2=(\Id, \frac{1}{2}(e_1 -e_2 -e_3))$ &      \\
    & $s_3=(\Id, \frac{1}{2}(e_1+e_2+e_3))$  &      \\
 \hline
\end{tabular}
\caption{Bieberbach groups with holonomy $\Z_4 \times \Z_2 $. See pages 59 and 60 in \cite{lambert}.}
\label{table:8}
\end{table}

\begin{table}
\begin{tabular}{|l|c|c|}
\hline
$\mathbf{N^4_{42}}$ &    $\pi = \langle t_1, t_2, t_3, s_1, \alpha, \beta \rangle $       &  $\alpha  = \left( A= \begin{psmallmatrix}
-E_0 & 0  \\
 0 & R(\frac{4\uppi}{3})
  \end{psmallmatrix} , \frac{1}{6}e_2 \right) $\\ 
      &  $s_1 = ( \Id , -A(e_3) )$  &  $\beta= \left( B= \begin{psmallmatrix}
\Id & 0  \\
 0 & -\Id
  \end{psmallmatrix} , \frac{1}{2}e_1  \right)$  \\
 \hline
\end{tabular}
\caption{Bieberbach groups with holonomy $\Z_6 \times \Z_2$. See page 66 in \cite{lambert}.}
\label{table:9}
\end{table}

In order to study the moduli space of flat metrics for these manifolds, we prove that the action of $\mathcal{N}_{\pi}$ splits into the factors of the Teichm\"uller space. First we prove that $\mathcal{N}_{\pi}$ can be separated as a product. 

\begin{remark} \label{shape}
We are studying a double quotient   
\begin{multline*}
\ort(4) \backslash \{ G \in \G(4, \R) \mid G H_{\pi} G^{-1} \subset \ort(4) \} / \mathcal{N}_{\pi} \cong  \\  \left( \ort(1) \backslash \G(1, \R) \times \ort(1) \backslash \G(1, \R) \times \mathrm{U}(1) \backslash \G(1, \C) \right) / \mathcal{N}_{\pi} .
\end{multline*}
Then we have the action of $\mathcal{N}_{\pi}$ on matrices of the form 
\[
G = \begin{psmallmatrix}
a  &  0  & 0 & 0 \\
0  &  b  & 0 & 0 \\
0  &  0 & c  & -d \\
0  &  0 & d  & c
  \end{psmallmatrix} \in \G(1, \R) \times \G(1, \R) \times \G(1, \C),   
\]
since for any $X \in \mathcal{N}_{\pi}$, we have that there exists a corresponding $ O \in \ort(4)$ such that $OGX \in \G(1, \R) \times \G(1, \R) \times \G(1, \C) $.
\end{remark}

The $4$-dimensional closed flat manifolds we are considering in this section have holonomy $H_{\pi} =  \langle A, B \rangle  $, where a corresponding rotation $R(\theta)$ is in $A$, with $R(\theta)^k = \Id$ for $k=3, 4, 6$. The only case where the order of $A$ is different from that of the corresponding rotation is $N^4_{42}$.
 
\begin{lemma} \label{separate}
Let $\pi$ be a Bieberbach group for the $4$-dimensional closed flat manifolds with one $1$-dimensional isotypic component of complex type and two generators in their holonomy, $H_{\pi}= \langle A , B \rangle$, where $A$ is of order $k=3, 4,$ or $6$, and $B$ is of order $2$. Each $X \in \mathrm{N}_{\G(4,\R)}(H_{\pi})$ has the form $X= \begin{psmallmatrix}
X_1 & 0    \\
0  &  X_2 \\  \end{psmallmatrix}$ with $X_1 , X_2 \in \G(2, \R)$.  
\end{lemma}


\begin{proof}
The decomposition of the orthogonal $H_{\pi}$-representation into irreducible subspaces is $\R^4 = V_1 \oplus V_2 \oplus V_3 $, where End$_{H_{\pi}}(V_1) \cong \R \cong$ End$_{H_{\pi}}(V_2)$ and End$_{H_{\pi}}(V_3) \cong \C$. The action of $\mathrm{N}_{\G(4,\R)}(H_{\pi})$ is as in remark \ref{shape}. Proceeding by contradiction, suppose there exists $X \in \mathrm{N}_{\G(4,\R)}(H_{\pi})$ such that $X \neq \begin{psmallmatrix}
X_1 & 0    \\
0  &  X_2 \\  \end{psmallmatrix}$ with $X_1 , X_2 \in \G(2, \R)$, then $X$ must mix the $\R$ type components with the $\C$ type components, which means that the subspaces $V_i$, $i\in \{1,2,3\}$, are not invariant under the action of $X$. 
Then there exists $p \in V_1 \oplus V_2$, $p\ne 0$, such that $X(p) \in V_3$ and $XA(p) = \tilde{A}X(p) \in V_3$ for some $A, \tilde{A} \in H_{\pi}$. We know that the action of $ H_{\pi}$ on $V_3$ is by rotations, $R(\theta)$, then $\{ X(p), R(\theta)X(p) \}$ forms a basis in $V_3$. By abuse of notation, we have  $R(\theta )(X(p)) = X(\hat{A}(p))$, for some $\hat{A} \in H_{\pi} $, where we restrict the action of the matrices to the corresponding components. We  conclude that $\{ p, \hat{A}(p) \}$ is a basis for $V_1 \times V_2$ and then the restriction of $X$ to $V_1 \times V_2$ takes values in $V_3$.
In order for this to happen, we have that
\[
X = \begin{psmallmatrix}
0 & X_1    \\
X_2  & 0  \\  \end{psmallmatrix} \quad \text{with} \quad X_1 , X_2 \in \G(2, \R),
\]              
but from the property of $XA(x,y,0,0)^{t} = \tilde{A}X(x,y,0,0)^{t}$, we get that $X_2 E (x,y)^{t} = R(\theta)X_2 (x,y)^{t}$, for $E \in \{ \begin{psmallmatrix}
\pm 1 & 0  \\
 0 & \pm 1
  \end{psmallmatrix} \}$, and then $E=X_2^{-1} R(\theta) X_2$, but this is a contradiction because conjugation preserves the order. 
\end{proof}

In the following proposition, we study the matrix part of the normalizer $\mathcal{N}_{\pi}$ for the $4$-dimensional closed flat manifolds with one $1$-dimensional isotypic component of $\C$ type. 

\begin{proposition} \label{normalizer}
Let $\pi$ be a Bieberbach group for the $4$-dimensional closed flat manifolds with two generators in their holonomy, $H_{\pi}= \langle A , B \rangle$, where  $A$ is of order $k=3,4$ or $6$, and $B$ of order $2$. The matrix part of the normalizer of $\pi$ in $\A(4)$ is either 
\[
\mathcal{N}_{\pi} = \left\langle \begin{psmallmatrix}
\pm 1 & 0 & 0    \\
  0  & 1  & 0    \\ 
  0  & 0  & R(\theta) \\  \end{psmallmatrix}, \begin{psmallmatrix}
\pm 1 & 0 & 0    \\
  0  & -1 & 0   \\
  0  &  0 & R(\theta)E_0 \\  \end{psmallmatrix} \right\rangle \quad \text{or} \quad  \left\langle \begin{psmallmatrix}
  1 & 0 & 0    \\
  0  & \pm 1  & 0    \\ 
  0  & 0  & R(\theta) \\  \end{psmallmatrix}, \begin{psmallmatrix}
 -1 & 0 & 0    \\
  0  & \pm 1 & 0   \\
  0  &  0 & R(\theta)E_0 \\  \end{psmallmatrix} \right\rangle , 
\]
\[
\quad \text{or} \quad \left\langle \begin{psmallmatrix}
R(\theta) & 0 & 0 \\ 
0  & \pm 1 & 0     \\
0  &  0  &  1      \\  \end{psmallmatrix}, \begin{psmallmatrix}
R(\theta)E_0 & 0 & 0 \\
0  & \pm 1 & 0     \\
0  &  0  &  -1      \\  \end{psmallmatrix} \right\rangle ,
\]
for all the cases but $O^4_{20}$, $O^4_{24}$, $N^4_{40}$ and $N^4_{41}$. For these cases we know that 
\[
\langle A , B \rangle \subset \mathcal{N}_{\pi} \subset \left\langle \begin{psmallmatrix}
\pm 1 & 0 & 0    \\
  0  & \pm 1  & 0    \\ 
  0  & 0  & R(\theta) \\  \end{psmallmatrix}, \begin{psmallmatrix}
\pm 1 & 0 & 0    \\
  0  & \pm 1 & 0   \\
  0  &  0 & E_0 \\  \end{psmallmatrix} \right\rangle \quad \text{or} \quad  \left\langle \begin{psmallmatrix}
R(\theta) & 0 & 0 \\ 
0  & \pm 1 & 0     \\
0  &  0  &  \pm 1      \\  \end{psmallmatrix}, \begin{psmallmatrix}
E_0 & 0 & 0 \\
0  & \pm 1 & 0     \\
0  &  0  &  \pm 1      \\  \end{psmallmatrix} \right\rangle , 
\]
\[
\quad \text{or} \quad \left\langle  \begin{psmallmatrix}
\left\langle - \Id, \; \begin{psmallmatrix} 0 & 1 \\
1 & 0 \end{psmallmatrix} \right\rangle  & 0    \\
 0  & R(\theta) \\  \end{psmallmatrix}, \begin{psmallmatrix}
\left\langle - \Id, \; \begin{psmallmatrix} 0 & 1 \\
1 & 0 \end{psmallmatrix} \right\rangle & 0    \\
0 & E_0 \\  \end{psmallmatrix} \right\rangle . 
\] 
\end{proposition}

\begin{proof}
Notice that in general we have $ \langle \alpha , \beta \rangle \subset \mathrm{N}_{\A(4)}(\pi) $; then,  $\langle A , B \rangle \subset \mathcal{N}_{\pi}$. \\ First, we prove that $\mathcal{N}_{\pi}$ is contained in 
$
\left\langle \begin{psmallmatrix}
\pm 1 & 0 & 0    \\
  0  & \pm 1  & 0    \\ 
  0  & 0  & R(\theta) \\  \end{psmallmatrix}, \begin{psmallmatrix}
\pm 1 & 0 & 0    \\
  0  & \pm 1 & 0   \\
  0  &  0 & E_0 \\  \end{psmallmatrix} \right\rangle$, or \\
  $  \left\langle \begin{psmallmatrix}
R(\theta) & 0 & 0 \\ 
0  & \pm 1 & 0     \\
0  &  0  &  \pm 1      \\  \end{psmallmatrix}, \begin{psmallmatrix}
E_0 & 0 & 0 \\
0  & \pm 1 & 0     \\
0  &  0  &  \pm 1      \\  \end{psmallmatrix} \right\rangle $, or $ \left\langle  \begin{psmallmatrix}
\left\langle - \Id, \; \begin{psmallmatrix} 0 & 1 \\
1 & 0 \end{psmallmatrix} \right\rangle  & 0    \\
 0  & R(\theta) \\  \end{psmallmatrix}, \begin{psmallmatrix}
\left\langle - \Id, \; \begin{psmallmatrix} 0 & 1 \\
1 & 0 \end{psmallmatrix} \right\rangle & 0    \\
0 & E_0 \\  \end{psmallmatrix} \right\rangle . 
$ 

Since the order is invariant under conjugation, we have that $XAX^{-1}= A$ or $A^r$, where $(k,r)=1$, and $XBX^{-1} \in \langle A \rangle \cdot B$ is another reflection. For $N^4_{27}$, $N^4_{28}$, $N^4_{29}$, and  $N^4_{42}$, we have the only possibility  $XBX^{-1}=B$.

By lemma \ref{separate}, $X$ has the form  
\[
X= \begin{psmallmatrix}
X_1 & 0    \\
0  &  X_2 \\  \end{psmallmatrix}, \quad \text{with} \quad X_1 , X_2 \in \G(2, \R).
\]
In most of the cases the rotation of the generator $A$ is in the second entry of the diagonal, so we will consider w.l.o.g. that the matrix $X_2$ satisfies
$$ X_2R(\theta)X_2^{-1}= \begin{cases}
R(\theta),    \\
R(\theta)^r. 
\end{cases}
$$
For the case of $N^4_{42}$, the order of $R(\theta)$ is $k/2$ and $r$ satisfies $(k/2,r)=1$. 

If $ X_2R(\theta)X_2^{-1}= R(\theta)$, $X_2$ must be a rotation preserving the corresponding lattice of $\pi$; then 
\[
X_2 \in \langle R(\theta ) \rangle \quad \text{and if} \quad k=3 \quad X_2 \in \langle R(\pi/3 ) \rangle .
\]    

If $ X_2R(\theta)X_2^{-1}= R(\theta)^r$, $X_2$ must be a reflection preserving the corresponding lattice of $\pi$; then 
\[
X_2 \in \langle R(\theta ) \rangle \cdot E_0 \quad \text{and if} \quad k=3 \quad X_2 \in \langle R(\pi/3 ) \rangle \cdot E_0 .
\]    

$X_1$ must preserve $E_0$; therefore it is a diagonal matrix. When we have the standard lattice in the first two entries $e_1, e_2 \in L_{\pi}$, then 
\[
X_1 = \begin{psmallmatrix}
\pm 1 & 0     \\
 0  &  \pm 1      \\  \end{psmallmatrix} .
\]  
We have some cases with just one standard entry $e_1 \in L_{\pi}$; in this case, we have to compute the type of matrices that preserve the lattice $L_{\pi}$ (see \cite{karla}), where we get that the first entry of this type of matrices is an integer. We conclude that $X_1$ is as before. For the case of $N^4_{41}$, we have to preserve $\begin{psmallmatrix}
0 & 1     \\
1  &  0      \\  \end{psmallmatrix}$; then,
\[
X_1 \in \left\langle - \Id, \; \begin{psmallmatrix} 0 & 1 \\
1 & 0 \end{psmallmatrix} \right\rangle .
\]

Thus 
\[
\mathcal{N}_{\pi} \subset \left\langle \begin{psmallmatrix}
\pm 1 & 0 & 0    \\
  0  & \pm 1  & 0    \\ 
  0  & 0  & R(\theta) \\  \end{psmallmatrix}, \begin{psmallmatrix}
\pm 1 & 0 & 0    \\
  0  & \pm 1 & 0   \\
  0  &  0 & E_0 \\  \end{psmallmatrix} \right\rangle \quad \text{or} \quad  \left\langle  \begin{psmallmatrix}
\left\langle - \Id, \; \begin{psmallmatrix} 0 & 1 \\
1 & 0 \end{psmallmatrix} \right\rangle  & 0    \\
 0  & R(\theta) \\  \end{psmallmatrix}, \begin{psmallmatrix}
\left\langle - \Id, \; \begin{psmallmatrix} 0 & 1 \\
1 & 0 \end{psmallmatrix} \right\rangle & 0    \\
0 & E_0 \\  \end{psmallmatrix} \right\rangle  . 
\]

From now on, we do not consider anymore the cases of $O^4_{20}$, $O^4_{24}$, $N^4_{40}$ and $N^4_{41}$. Finally, we prove 
\[
\mathcal{N}_{\pi} = \left\langle \begin{psmallmatrix}
\pm 1 & 0 & 0    \\
  0  & 1  & 0    \\ 
  0  & 0  & R(\theta) \\  \end{psmallmatrix}, \begin{psmallmatrix}
\pm 1 & 0 & 0    \\
  0  & -1 & 0   \\
  0  &  0 & R(\theta)E_0 \\  \end{psmallmatrix} \right\rangle \quad \text{or} \quad  \left\langle \begin{psmallmatrix}
  1 & 0 & 0    \\
  0  & \pm 1  & 0    \\ 
  0  & 0  & R(\theta) \\  \end{psmallmatrix}, \begin{psmallmatrix}
 -1 & 0 & 0    \\
  0  & \pm 1 & 0   \\
  0  &  0 & R(\theta)E_0 \\  \end{psmallmatrix} \right\rangle , 
\]
\[
\quad \text{or} \quad \left\langle \begin{psmallmatrix}
R(\theta) & 0 & 0 \\ 
0  & \pm 1 & 0     \\
0  &  0  &  1      \\  \end{psmallmatrix}, \begin{psmallmatrix}
R(\theta)E_0 & 0 & 0 \\
0  & \pm 1 & 0     \\
0  &  0  &  -1      \\  \end{psmallmatrix} \right\rangle .
\]

We have to analyze case by case.
In most of the cases, the translation part of the generators looks like 
    \[
\alpha = \left(A, v=(v_1, \frac{1}{k}, v_3, v_4) \right) \quad \text{and} \quad \beta = \left(B, u=( \frac{1}{2}, \frac{1}{k}, 0, 0) \; \mathrm{or} \; u=(u_1, u_2, u_3, u_4) \right), 
\]
where  $v_i \in \{ 0, \frac{1}{2}, \frac{1}{4} \}$ and $u_i \in \{ 0, \frac{1}{2} \} $.
To preserve the lattice of $\alpha$ we need $X_1 = \begin{psmallmatrix}
\pm 1 & 0     \\
 0  &  1      \\  \end{psmallmatrix} $. We know that $\alpha =(A,v) \in \mathrm{N}_{\A(4)}(\pi)$. When considering $P= \begin{psmallmatrix}
-1 & 0 &  0  \\
 0 & 1 &  0  \\
 0 & 0 & R(\theta )    \end{psmallmatrix}$, we use the same translation $v$ to see that $\tau =(P,v)$ is also in $\mathrm{N}_{\A(4)}(\pi)$. This also works for the case when the translation part of the generators is as before but the value $\frac{1}{k}$ of $v$ and $u$ is in the first entry; we just consider $X_1 = \begin{psmallmatrix}
1 & 0     \\
 0  & \pm 1      \\  \end{psmallmatrix} $.  
For the case of $N^4_{29}$, we get that $ \left( \begin{psmallmatrix}
R(\theta ) & 0 & 0 \\
0 &  -1 & 0  \\
0 &  0  &  1   \end{psmallmatrix}, ( \frac{1}{4}, -\frac{1}{4}, -\frac{1}{4}, 0) \right) \in \mathrm{N}_{\A(4)}(\pi) $. 
 
To switch to the generator $A^r$ we need $X_1 = \begin{psmallmatrix}
\pm 1 & 0     \\
 0  &  -1      \\  \end{psmallmatrix} $ or $X_1 = \begin{psmallmatrix}
- 1 & 0     \\
 0  &  \pm 1      \\  \end{psmallmatrix} $ if the value $\frac{1}{k}$ of $v$ and $u$ is in the second or first entry. We have the following cases depending on the generator $\beta =(B,u)$. 
 \begin{itemize}
     \item When the matrix $B$ has a reflection of the form $R(\theta )^{s_1}E_0$, for some integer $s_1$. Then we use directly that $\beta \in \mathrm{N}_{\A(4)}(\pi)$, and for $X_1= -\Id$ we use that $\alpha^{s_2} \tau \beta \in \mathrm{N}_{\A(4)}(\pi)$ for any integer  $s_2$.
     \item When $B$ has no reflections, we  search for a suitable translation $x \in \R^4 $ such that $\left( \begin{psmallmatrix}
1 & 0  & 0  \\
0  &  -1  & 0 \\ 
0  &   0  & R(\theta )E_0 \end{psmallmatrix}, x \right) \in \mathrm{N}_{\A(4)}(\pi) $.  \\ 
For $N^4_{27}$, $N^4_{28}$ and $N^4_{42}$ we get that $\left( \begin{psmallmatrix}
1 & 0  & 0  \\
0  &  -1  & 0 \\ 
0  &   0  & R(\theta )E_0 \end{psmallmatrix}, 0 \right) \in \mathrm{N}_{\A(4)}(\pi) $. \\
For $N^4_{29}$ we obtain that 
$ \left( \begin{psmallmatrix}
R(\theta )E_0 & 0 & 0 \\
0 &  1 & 0  \\
0 &  0  &  -1    \end{psmallmatrix}, ( \frac{1}{4}, \frac{1}{4}, -\frac{1}{4}, 0) \right) \in \mathrm{N}_{\A(4)}(\pi)$. \qedhere
\end{itemize}
\end{proof}

          
\section{Subgroups of SL(2,$\mathbb{Z}$)} \label{sec:subgroups}

In this section we study the topology of the remaining cases, where we have double quotients of the form
\begin{equation} \label{doublequot}
\ort(2) \backslash \G(2, \R) / \Gamma , \quad \text{with} \; \Gamma \leq \G(2,\mathbb{Z}).
\end{equation}

The Teichm\"uller space of flat metrics 
on the $2$-torus is well understood, see \cite{mapclassgr}; in the cited reference it is shown the existence of a homeomorphism 
\begin{equation} \label{homeohyp}
\ort(2) \backslash \G(2, \R) \cong \R^+ \times \mathbb{H}^2 .
\end{equation}
To study the double quotient \eqref{doublequot}, we analyze the action of $\Gamma$ on the hyperbolic plane. First of all observe that we can just take the matrices with positive determinant, denoted by $\Gamma^+$. The action of matrices in $\mathrm{SL}(2, \mathbb{Z})$ on $\mathbb{H}^2$ is via M\"obius transformations; and using information such as the fundamental domain of $\mathrm{SL}(2, \mathbb{Z})$ on $\mathbb{H}^2$ and the index of $\Gamma^+$ in $\mathrm{SL}(2, \mathbb{Z})$, we can compute the fundamental domain of $\Gamma^+$. See \cite{karla}.

Here are some results of double quotients studied in \cite{karla}.

\begin{itemize}
  \item$\mathrm{O}(2)\backslash \G(2,\mathbb{R}) /\G(2,\mathbb{Z}) \cong \R^+ \times (\mathbb{H}^2 /\mathrm{SL}(2,\mathbb{Z})) 
 \cong \R^+ \times (\mathbb{S}^2 \setminus \{*\}) \cong \R^3 $ .
 \item $ \ort(2) \backslash \G(2, \R) / \Gamma(2)\cdot \left\langle Y \right\rangle \cong \R^+ \times \mathbb{H}^2 / \Gamma(2)\cdot \left\langle Y \right\rangle^+ \cong \R^+ \times \mathbb{S}^1\times \R$, where $Y=\begin{psmallmatrix}
 0 & -1  \\
 1  & 0
  \end{psmallmatrix}$. 
 \item $ \ort(2) \backslash \G(2, \R) / \Gamma(2) \cong \R^+ \times \mathbb{H}^2/ \Gamma(2)^+ \cong \R^+ \times 3\text{-punctured sphere} $.  
\end{itemize}

\begin{figure*}[h]
\begin{multicols}{3}
    \includegraphics[width=0.6\linewidth]{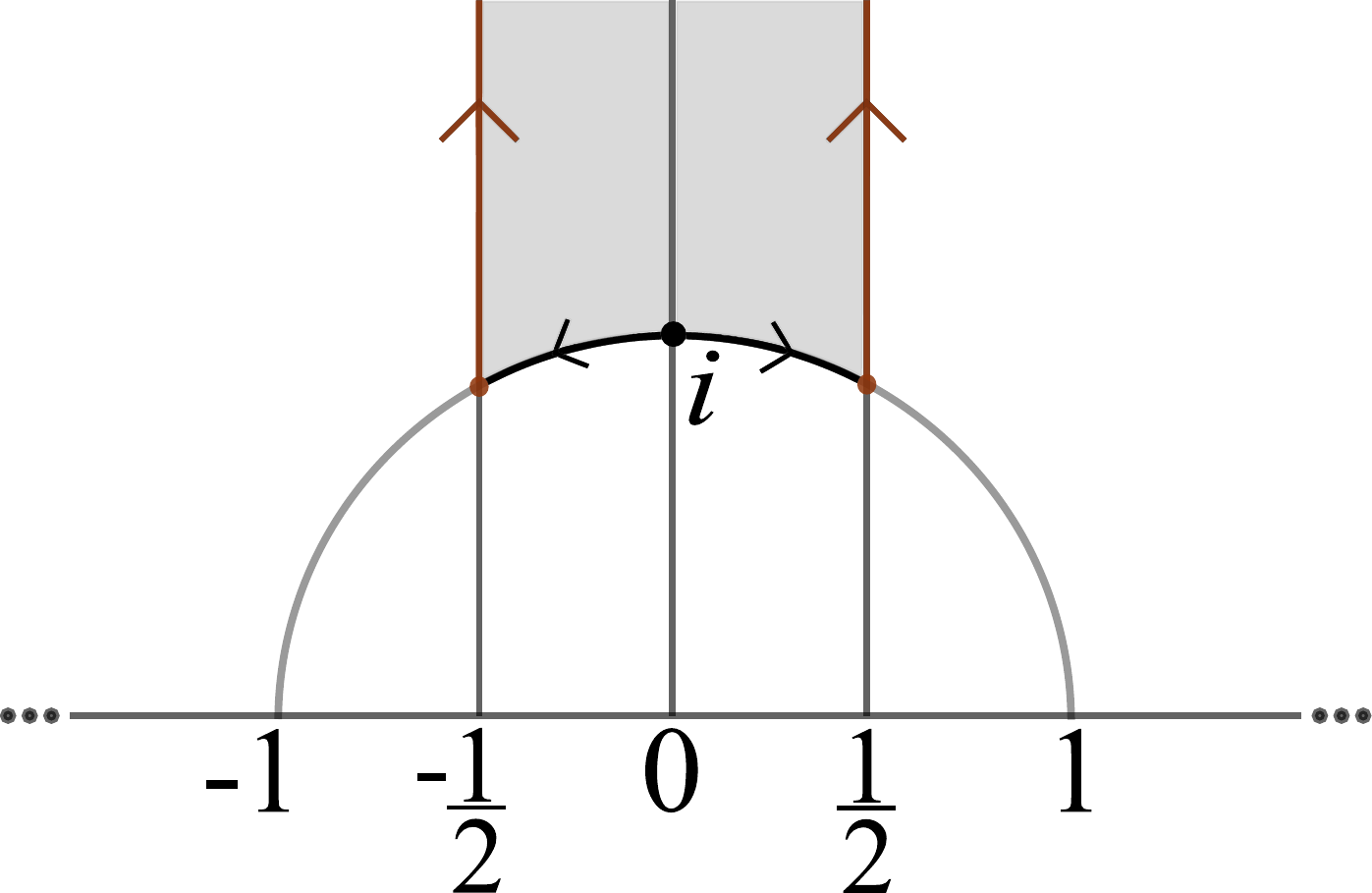}\par 
    \includegraphics[width=0.6\linewidth]{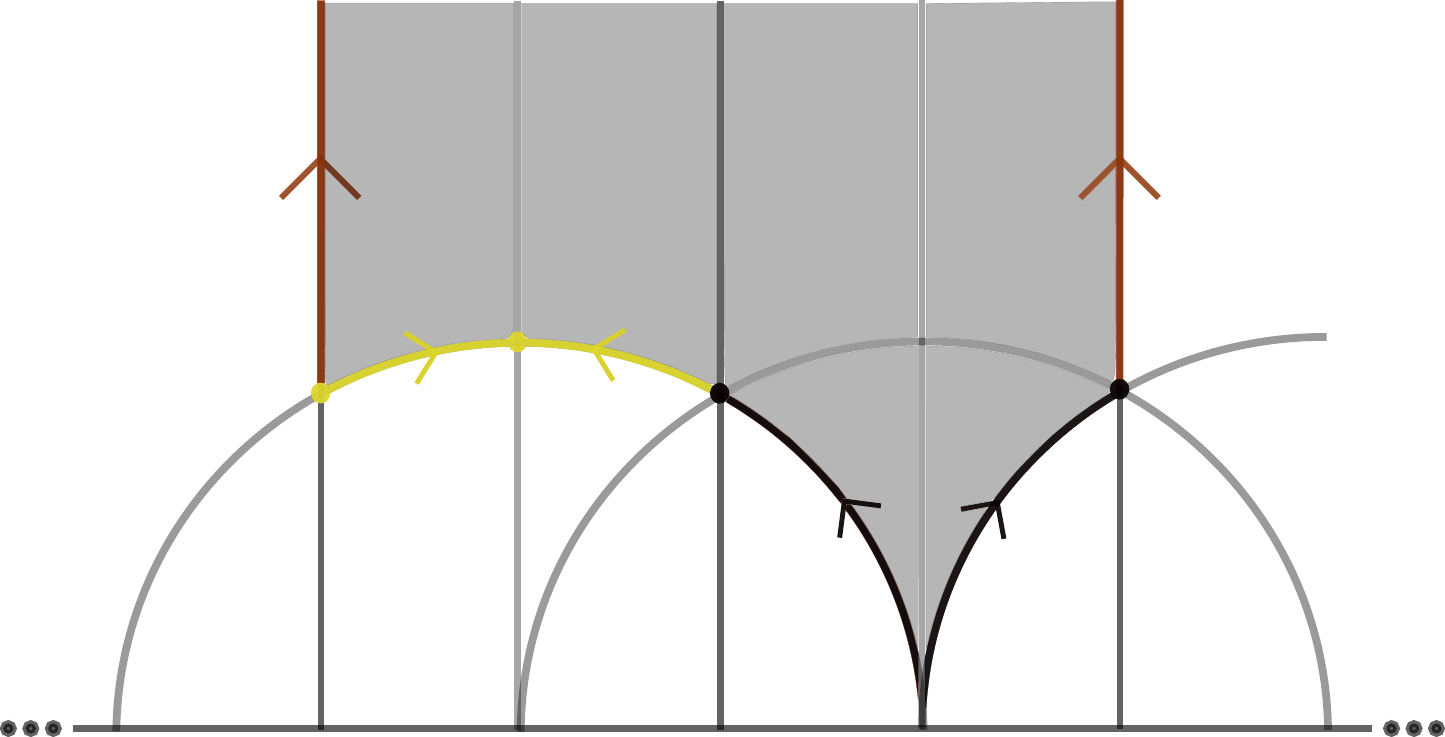}\par 
    \includegraphics[width=0.6\linewidth]{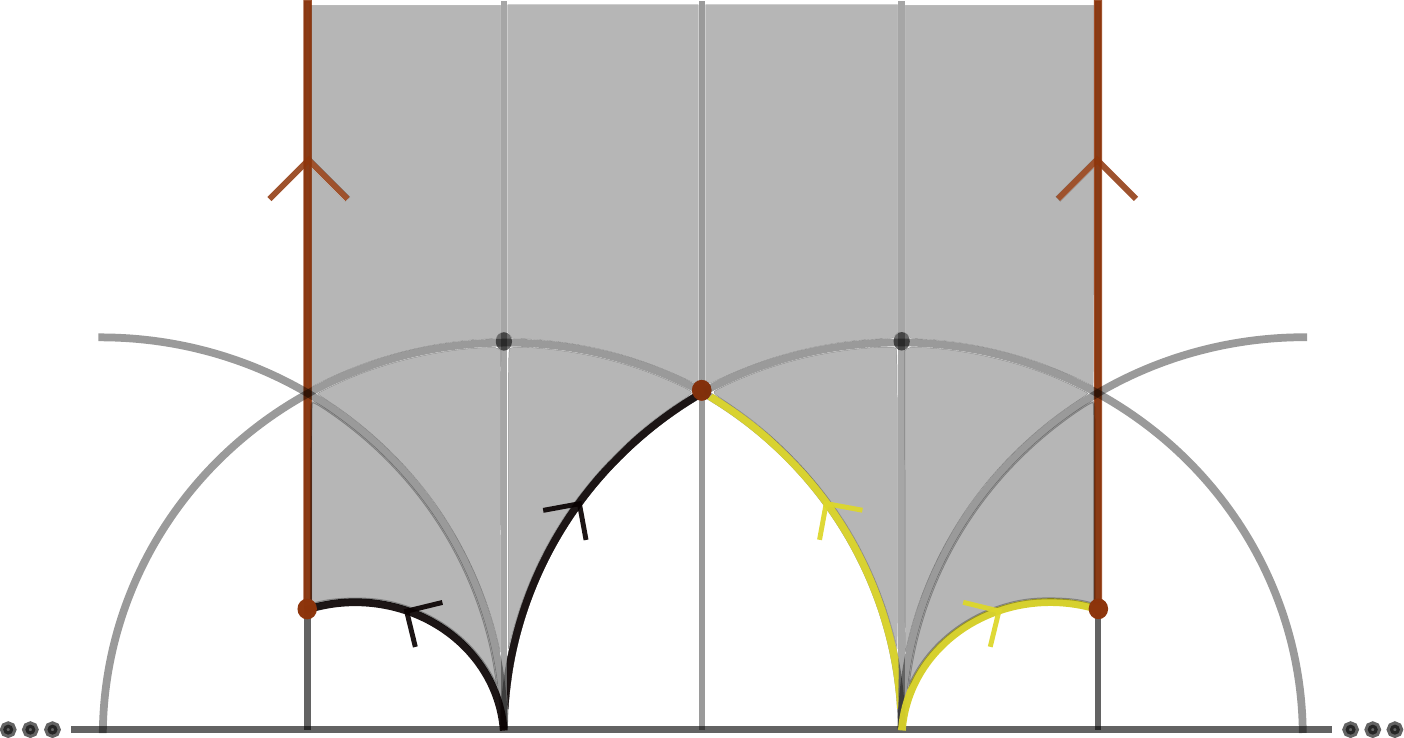}\par
    \end{multicols}
\caption{Fundamental domains of $\mathrm{SL}(2,\mathbb{Z})$, $\Gamma(2)\cdot \left\langle Y \right\rangle^+$ and $\Gamma(2)^+$ on $\mathbb{H}^2$, respectively.}
\label{3funddomains}
\end{figure*}

We compute the fundamental domain for the subgroups $\Gamma_0(2)^{t+}$ and $\Gamma_0(4)^+$ of $\mathrm{SL}(2, \Z )$. 

First, observe that $[\mathrm{SL}(2, \Z) : \Gamma_0(2)^{t+}] = 3$. This is computed with the same procedure as the case of $\Gamma_0(2)^{+}$ in \cite{karla}. The representatives we choose in $\mathrm{SL}(2, \Z)$ are \\
$\Id$,  $\left( \begin{array}{cc}
0 & -1 \\ 1 & 0 \end{array} \right)= S$,  $\left( \begin{array}{cc}
1 & 1 \\ 0 & 1 \end{array} \right)=T$, where $S$ and $T$ are the generators of $\mathrm{SL}(2, \Z)$. Thus, we get the fundamental domain of $\Gamma_0(2)^{t+}$ on $\mathbb{H}^2$, shown in figure \ref{fund2_0t}. The borders are identified by $T^2$, $\left( \begin{array}{cc}
1 & -2 \\ 1 & -1 \end{array} \right)$ and $\left( \begin{array}{cc}
-1 & 0 \\ -1 & -1 \end{array} \right)$ $\in \Gamma_0(2)^{t+}$. We conclude that $\mathbb{H}^2 / \Gamma_0(2)^{t+} \cong \mathbb{S}^1\times \R$.   

\begin{figure}[h]
\begin{center}
\scalebox{0.3}[0.3]{\includegraphics{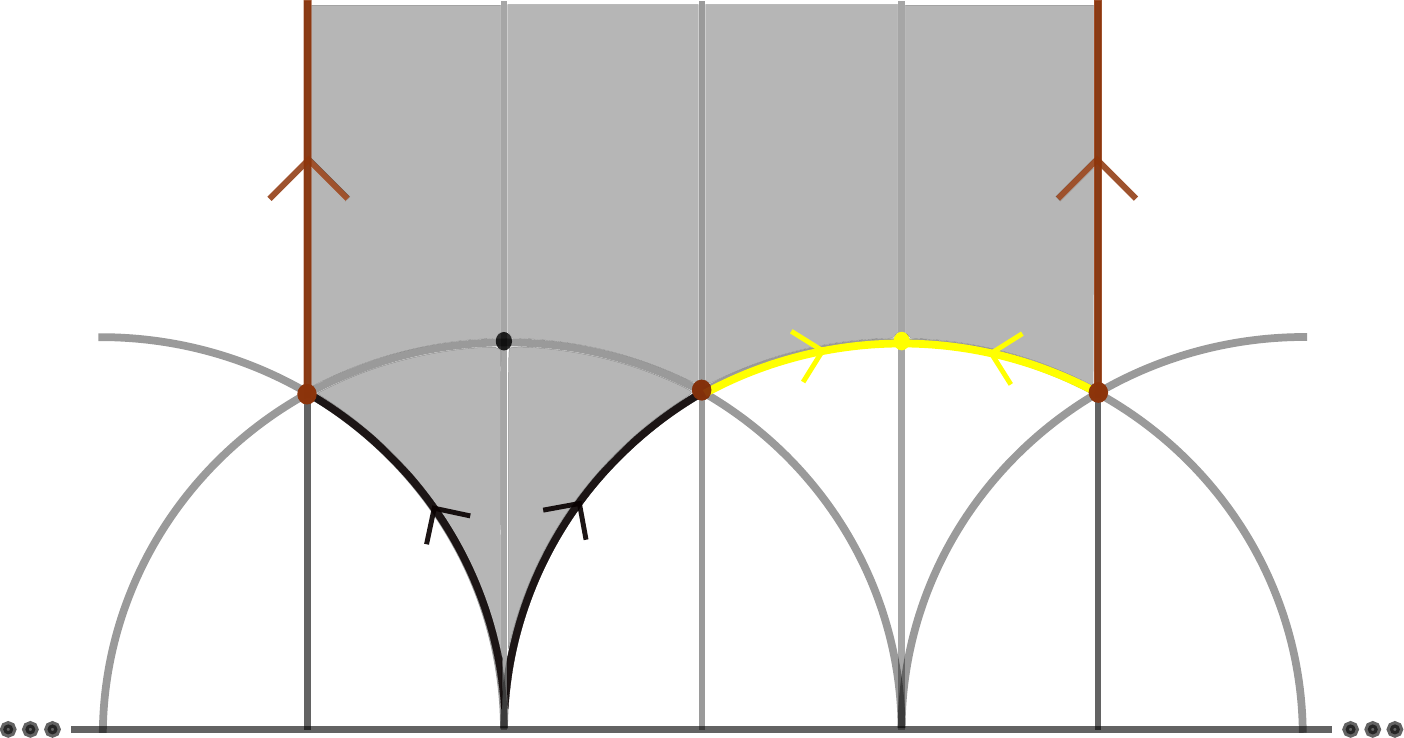}}
\caption{The fundamental domain of $\Gamma_0(2)^{t+}$ on $\mathbb{H}^2$.}
\label{fund2_0t}
\end{center}
\end{figure}

Now, we have that $[\mathrm{SL}(2, \Z) : \Gamma_0(4)^{+}] = 12$, because
\begin{equation*}
[\mathrm{SL}(2, \Z) : \Gamma_0(4)^{+}] = [\mathrm{SL}(2, \Z) : \Gamma(2)^{+}] [ \Gamma(2)^{+} : \Gamma_0(4)^{+}] = 6 \cdot 2.
\end{equation*}
To choose the representatives in $\mathrm{SL}(2, \Z)$ for $\Gamma_0(4)^{+}$, we use the representatives already given in \cite{karla} for $\Gamma(2)^+$, denoted here as $\gamma_i$ for $i\in \{1, \dots ,6\}$. We have that $\Gamma(2)^+= \Gamma_0(4)^{+} \alpha_1 \sqcup \Gamma_0(4)^{+} \alpha_2 $, where $\alpha_1= \Id$ and $\alpha_2 =\left( \begin{array}{cc}
1 & 0 \\ 2 & 1 \end{array} \right)$, then the representatives we choose are $\{ \gamma_1, \dots ,\gamma_6, \alpha_2\gamma_1, \dots, \alpha_2 \gamma_6 \}$. Thus, we get the fundamental domain of $\Gamma_0(4)^{+}$ on $\mathbb{H}^2$, shown in figure \ref{fund4_0}. The borders are identified by $T^{-2}$, $\left( \begin{array}{cc}
1 & 0 \\ 4 & 1 \end{array} \right)$, $\left( \begin{array}{cc}
-3 & 2 \\ -8 & 5 \end{array} \right)$, $\left( \begin{array}{cc}
3 & -2 \\ 8 & -5 \end{array} \right)$ and $\left( \begin{array}{cc}
1 & -2 \\ 4 & -7 \end{array} \right)$ $\in \Gamma_0(4)^{+}$. We conclude that $\mathbb{H}^2 / \Gamma_0(4)^{+} \cong 4$-punctured sphere.

\begin{figure}[h]
\begin{center}
\scalebox{0.4}[0.4]{\includegraphics{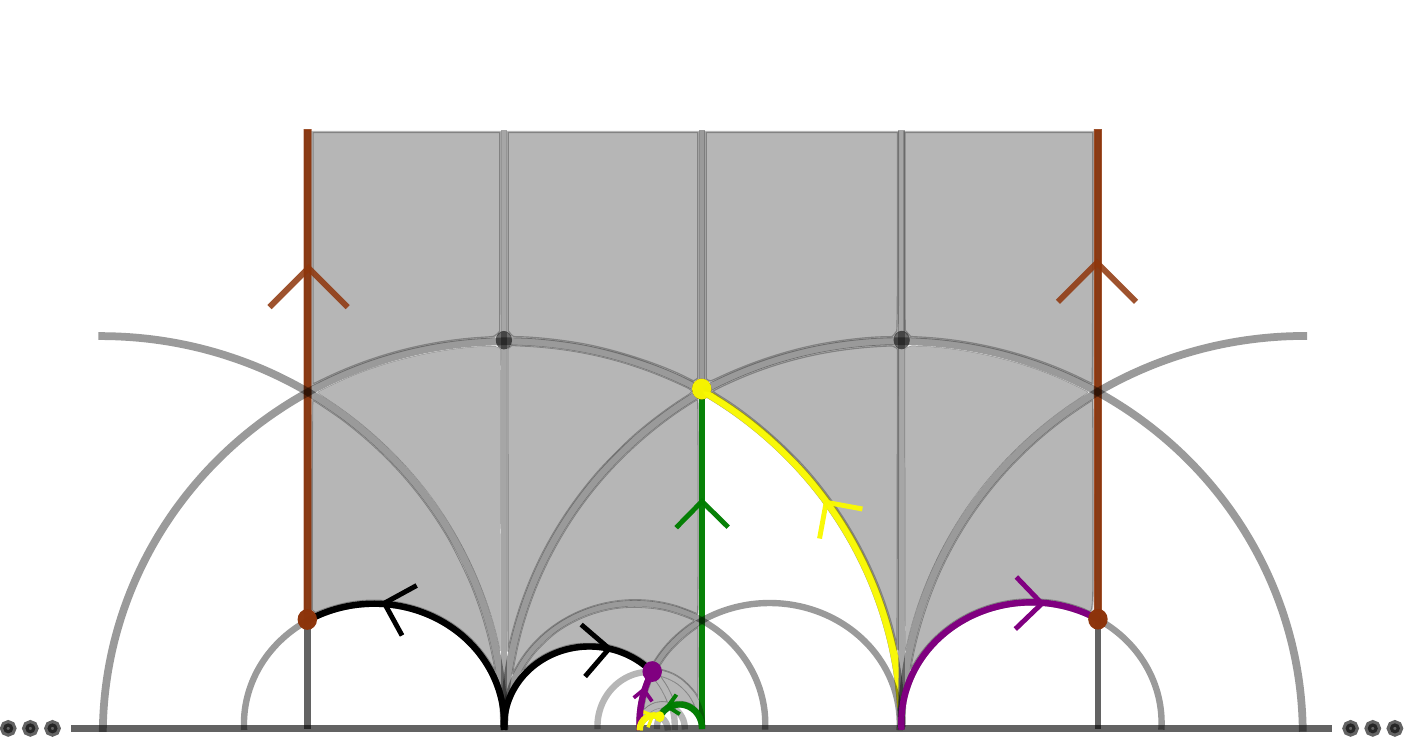}}
\caption{The fundamental domain of $\Gamma_0(4)^{+}$ on $\mathbb{H}^2$.}
\label{fund4_0}
\end{center}
\end{figure}

\section{The moduli space of flat metrics}\label{sec:moduli}

From the previous sections, we have all the information of $\mathcal{T}_{flat}(M)$ and  $\mathcal{N}_{\pi}$ for the $4$-dimensional closed flat manifolds with two or more generators in their holonomy. Then we reduce the quotient in order to prove the theorem.   

\begin{proof}[Proof of Theorem \ref{modulimoregen}]
We use the description $\mathcal{M}_{flat}(M) = \mathcal{T}_{flat}(M)/ \mathcal{N}_{\pi}$ from proposition \ref{moduli} and the information of the Teichm\"uller space from Proposition \ref{teich}.  

For the cases of four different $1$-dimensional $\R$ type isotypic components we have 
\[
\mathcal{M}_{flat}(M) = (\R^+)^4/ \mathcal{N}_{\pi}.
\]
We use Propositions \ref{case 1} and \ref{case 2}, where $\mathcal{N}_{\pi}$ is computed. We have three different types of actions by $\mathcal{N}_{\pi}$: 
\begin{itemize}
\item When $\mathcal{N}_{\pi}= \left\{ \mathrm{diag}(\pm 1, \pm 1, \pm 1, \pm 1)   \right\}$, we get $\mathcal{M}_{flat}(M) = (\R^+)^4$.
\item When we have one permutation between two entries, for example, $\mathcal{N}_{\pi}= \left\{ \mathrm{diag}(\pm 1, \pm 1, \pm 1, \pm 1), \begin{psmallmatrix}
0 & \pm 1 & 0 & 0  \\
\pm 1 & 0 & 0 & 0 \\
0 & 0 & \pm 1 & 0 \\ 
0 & 0 & 0     &  \pm 1     \end{psmallmatrix}  \right\}$, we get $ \mathcal{M}_{flat}(M) \cong (\R^{+})^3 \times [0, \infty  ) $.
\item When we have all the permutations of the three entries of $\R^3$, which is the case of $\mathbf{O^4_{14}}$ with $\mathcal{N}_{\pi}= \left\{ \begin{psmallmatrix}
\sigma &  0  \\
 0    &  \pm 1      \end{psmallmatrix} \mid \sigma \in \left\{  \mathrm{diag}( \pm 1, \pm 1, \pm 1) \rtimes \mathrm{S}_3  \right\} \right\}$, we get
 \begin{align*}
    \mathcal{M}_{flat}(M) & \cong  \R^+ \times (\R^+)^3 / S_3  \\ & \cong  \R^+ \times \{ (x,y,z) \in (0, \infty )^3 \mid x \leq y \leq z \} \\
    & \cong (\R^{+})^2 \times [0, \infty  )^2,
 \end{align*} 
where the last homeomorphism is given by $(x,y,z) \mapsto (x, y-x, z-y )$.
\end{itemize}

For the cases of two different $1$-dimensional $\R$ type isotypic components we have $\mathcal{M}_{flat}(M) = (\R^+)^2/ \mathcal{N}_{\pi}$, and from Proposition \ref{case 3} we get that the action of $\mathcal{N}_{\pi}$ is always reduced to the action of $\left\{ \mathrm{diag}(\pm 1, \pm 1 ) \right\}$. Thus we get $\mathcal{M}_{flat}(M) = (\R^+)^2$.

For the cases of two different $1$-dimensional $\R$ type isotypic components and one $2$-dimensional $\R$ type isotypic component we have 
\[
\mathcal{M}_{flat}(M) = \left( \frac{\G(2, \R)}{\ort(2)} \times (\R^+)^2 \right) \Big/ \mathcal{N}_{\pi}.
\]
We use Proposition \ref{case 4} to substitute the corresponding computation for $\mathcal{N}_{\pi}$. In most of these cases, we can just split the action of $\mathcal{N}_{\pi}$ into each of the factors. For $N^4_{3}$, $N^4_{5}$, $N^4_{6}$ and $N^4_{7}$, we can not split the action, but we can do the following. The double quotient looks 
\[
\left( \frac{\G(2, \R)}{\ort(2)} \times (\R^+)^2 \right) \Big/ \left\langle \begin{psmallmatrix}
\Gamma_1 & 0 & 0    \\ 
0  &  \pm 1 & 0   \\
0  &   0    & \pm 1  \end{psmallmatrix}, \begin{psmallmatrix}
\Gamma_2 & 0 & 0   \\ 
0  &  0  &  1  \\
0  &  1  &  0   \end{psmallmatrix} \right\rangle,
\]
where $\Gamma_1$ is a group and $\Gamma_2$ is not. For any $B \in \Gamma_2$ there exists $A \in \Gamma_1$ and a fixed $B_0 \in \Gamma_2$ such that $B=AB_0$. Then the analysis of the action of $\Gamma_2$ is reduced to understand the action of $B_0$. The double quotient is reduced as follows: 
\begin{multline*}
    \left( \frac{\G(2, \R)}{\ort(2)} \times (\R^+)^2 \right) \Big/ \left\langle \begin{psmallmatrix}
\Gamma_1 & 0 & 0    \\ 
0  &  \pm 1 & 0   \\
0  &   0    & \pm 1  \end{psmallmatrix}, \begin{psmallmatrix}
\Gamma_2 & 0 & 0   \\ 
0  &  0  &  1  \\
0  &  1  &  0   \end{psmallmatrix} \right\rangle \cong  \\
\ort(2) \backslash \G(2, \R) / \Gamma_1 \times (\R^+)^2 \Big/ \left\langle \begin{psmallmatrix}
B_0 & 0 & 0    \\ 
0  &  0 & 1   \\
0  &  1 & 0  \end{psmallmatrix} \right\rangle. 
\end{multline*}
The action of $B_0$ can be analyzed by the map 
\[
B_0: \ort(2) \backslash \G(2, \R) / \Gamma_1 \rightarrow \ort(2) \backslash \G(2, \R) / \Gamma_1, \quad \text{where} \quad [[G]] \mapsto [[GB_0]]. 
\]

For the following, we consider the two cases separately:
\begin{itemize}
    \item For $N^4_{3}$, $N^4_{5}$ and $N^4_{6}$, where $B_0= \left( \begin{array}{cc}
0 & -1 \\ 1 & 0 \end{array} \right)$. Since $\Gamma_1 = \Gamma(2)$, we know that
\[
\ort(2) \backslash \G(2, \R) / \Gamma_1 \cong \R^+ \times \mathbb{H}^2/ \Gamma(2)^+ . 
\]
The map $B_0$ is a reflection with $i$ being the only fixed point on $\mathbb{H}^2$. We already know the fundamental domain of $\Gamma(2)^+$ and $\left\langle \Gamma_1 , \Gamma_2 \right\rangle = \Gamma(2)\cdot \left\langle Y \right\rangle^+$ on $\mathbb{H}^2$ thanks to \cite{karla} and shown in section \ref{sec:subgroups}, figure \ref{3funddomains}, then it is easier to see how $B_0$ acts. 
\begin{figure}[h]
\begin{center}
\scalebox{0.3}[0.3]{\includegraphics{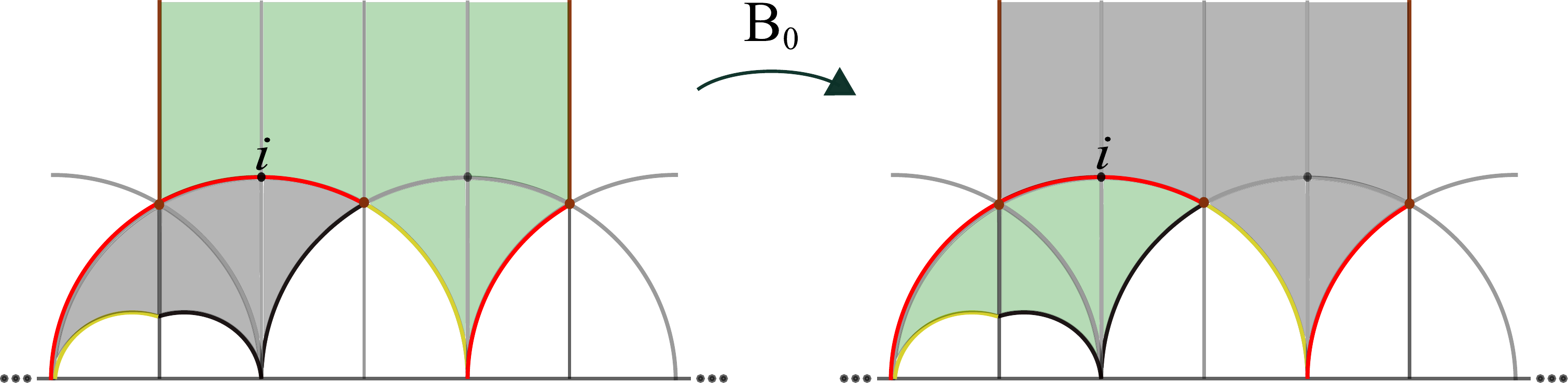}}
\caption{The action of $B_0$ on  $\mathbb{H}^2/ \Gamma(2)^+$.}
\label{fund2B_0}
\end{center}
\end{figure}
As shown in the figure \ref{fund2B_0}, the map $B_0$ sends the green domain to the green one and the grey to the grey one.

Let $a, b \in \R^+$ such that $a<b$. Then we have that $([[G]],a,b)$ is  related only to $([[GB_0]],b,a)$. When we pick the  representatives $\{ (a,b) \in (\R^+)^2 \mid a \leq b \}$ and $a<b$, the two cylinders are glued together by the red path shown in figure \ref{fund2B_0}. In other words, let $f_{b-a}$ be the identity map restricted to the red path. We obtain the attached space $X_1 \cup _{f_{b-a}} X_2$ with $X_1=X_2 = \mathbb{H}^2 / \Gamma(2)\cdot \left\langle Y \right\rangle^+ $. When $a=b$, we have the attached space $X_1 \cup _{f_{0}} X_2 \cong X_1$, where $f_0$ is the identity in the whole space $X_1$.     

    \item For $N^4_{7}$, where $B_0= \left( \begin{array}{cc}
1 & 0 \\ 2 & 1 \end{array} \right)$. Since $\Gamma_1 = \Gamma_0(4)$, we know that
\[
\ort(2) \backslash \G(2, \R) / \Gamma_1 \cong \R^+ \times \mathbb{H}^2/ \Gamma_0(4)^+.
\]
The map $B_0$ is also a reflection and without fixed points in $\mathbb{H}^2$, obtaining a double covering. In this case, $\left\langle \Gamma_1 , \Gamma_2 \right\rangle = \Gamma(2)$, and as in the previous case, when $a,b \in \R^+$ and $a<b$ we have two copies of the space $\mathbb{H}^2 / \Gamma(2)^+$ glued together by a path. When $a=b$, we get only one copy of the $3$-punctured sphere. We get the path explicitly using the fundamental domains given in section \ref{sec:subgroups}, figures \ref{3funddomains} and \ref{fund4_0}.
\end{itemize}

For the case of one $1$-dimensional $\C$ type isotypic component and two $\R$ type isotypic components, we have that the double quotient is
\begin{align*}
\mathcal{M}_{flat}(M) & = \mathcal{T}_{flat}(M)/ \mathcal{N}_{\pi} \\
        & = \ort(1) \backslash \G(1, \R) \times \ort(1) \backslash \G(1, \R) \times \mathrm{U}(1) \backslash \G(1, \C)  \Big/ \mathcal{N}_{\pi}
\end{align*}
From Proposition \ref{normalizer} we know that 
\[
\left\langle \begin{psmallmatrix}
\Id & 0    \\ 
0  &  R(\theta ) \end{psmallmatrix}, \begin{psmallmatrix}
E_0 & 0   \\ 
0  & R(\theta )E_0 \end{psmallmatrix} \right\rangle \subset \mathcal{N}_{\pi} \subset \{ \begin{psmallmatrix}
\pm 1 & 0   \\
0  &  \pm 1    \end{psmallmatrix} \} \times \langle R(\theta), R(\theta)E_0   \rangle  .
\]
Even though $\mathcal{N}_{\pi}$ is not the same as the product of the right hand side, we can still split the double quotient because the action of $\{ \begin{psmallmatrix}
\pm 1 & 0   \\
0  &  \pm 1    \end{psmallmatrix} \}$ is the same as the action of $\ort(1) \times \ort(1)$. Then,
\begin{multline*}
\ort(1) \backslash \G(1, \R) \times \ort(1) \backslash \G(1, \R) \times \mathrm{U}(1) \backslash \G(1, \C)  \Big/ \mathcal{N}_{\pi} \cong \\
\ort(1) \backslash \G(1, \R) \times \ort(1) \backslash \G(1, \R) \Big/ \{ \pm 1 ,  \pm 1  \} \times  \mathrm{U}(1) \backslash \G(1, \C)  \Big/ \langle R(\theta), R(\theta)E_0  \rangle .
\end{multline*}
Each double quotient reduces to $\R^+$, obtaining $\mathcal{M}_{flat}(M) \cong  (\R^+)^3$.
\end{proof}

\begin{proof}[Proof of Corollary \ref{topmoregen}]

We use the homeomorphisms from section \ref{sec:subgroups}.
\end{proof}

\bibliographystyle{siam}

\end{document}